\documentclass[11pt]{article}

\usepackage{amssymb}
\usepackage{amsthm}
\usepackage{amsmath}
\usepackage{amscd}

\allowdisplaybreaks
    
\newcommand{\R}{\mathbb{R}} 

\newcommand{\Q}{\mathbb{Q}}

\newcommand{\Sph}{\mathbb{S}}

\newcommand{\eps}{\varepsilon}

\newcommand{\dif}{\mathrm{d}} 
\newcommand{\Dif}{\mathrm{D}}

\newcommand{\linspan}{\mathrm{span}}

\newcommand{\cobg}{C^{0,\beta}_{\delta - 2}}
\newcommand{\ctbg}{C^{2,\beta}_\delta}
\newcommand{\clbg}{C^{k,\beta}_\delta}
\newcommand{\cobgs}{C^{0,\beta}_{\delta - 2, \mathit{sym} }}
\newcommand{\ctbgs}{C^{2,\beta}_{\delta, \mathit{sym} }}

\newcommand{\btheta}{\boldsymbol{\Theta}}

\newcommand{\approxsol}{\tilde \Lambda[\alpha, \Gamma,  \tau, \vec \sigma]}
\newcommand{\difop}{\Phi_{\tau, \vec \sigma}}
\newcommand{\linop}{{\mathcal L}_{\tau, \vec \sigma}}
\newcommand{\rightinv}{{\mathcal R}_{\tau, \vec \sigma}}

\DeclareMathOperator{\arccosh}{arccosh}

\newtheorem{mainthm}{Main Theorem}
\newtheorem{thm}{Theorem}
\newtheorem{lemma}[thm]{Lemma}

\newtheorem{prop}[thm]{Proposition}
\newtheorem*{nonumthm}{Theorem}

\theoremstyle{definition}
\newtheorem{defn}[thm]{Definition}

\newtheoremstyle{rmk}{5pt}{5pt}{}{}{\scshape}{:}{.5em}{}
\theoremstyle{rmk}
\newtheorem*{rmk}{Remark}

\newcommand{\mylabel}
	{\label}
\begin{document}

\title{Constant Mean Curvature Hypersurfaces in $\Sph^{n+1}$ \\ by Gluing Spherical Building Blocks}

\author{
Adrian Butscher \\ University of Toronto at Scarborough \\ email: \ttfamily butscher@utsc.utoronto.ca
}
\maketitle

\begin{abstract}
	The techniques developed by Butscher in \cite{me5} for constructing constant mean curvature (CMC) hypersurfaces in $\Sph^{n+1}$ by gluing together spherical building blocks are generalized to handle less symmetric initial configurations.  The outcome is that the approximately CMC hypersurface obtained by gluing the initial configuration together can be perturbed into an exactly CMC hypersurface only when certain global geometric conditions are met.  These \emph{balancing conditions} are analogous to those that must be satisfied in the `classical' context of gluing constructions of CMC hypersurfaces in Euclidean space, although they are more restrictive in the $\Sph^{n+1}$ case.  An example of an initial configuration is given which demonstrates this fact; and another example of an initial configuration is given which possesses no symmetries at all.   
\end{abstract}

%\tableofcontents

\renewcommand{\baselinestretch}{1.25}
\normalsize

\section{Introduction}

\paragraph{Gluing constructions of constant mean curvature hypersurfaces.}
A constant mean curvature (CMC) hypersurface $\Lambda$ contained in an ambient Riemannian manifold $M$ of dimension $n+1$ has the property that its mean curvature with respect to the induced metric is constant.  This property ensures that $n$-dimensional area of $\Lambda$ is a critical value of the area functional for hypersurfaces of $M$ subject to an enclosed-volume constraint.  Constant mean curvature hypersurfaces have been objects of great interest since the beginnings of modern differential geometry.  One very important method for constructing CMC hypersurfaces is the \emph{gluing technique} in which a more complex CMC hypersurface is built up from simple CMC building blocks. This technique was pioneered by Kapouleas in the context of CMC hypersurfaces in $\R^3$ \cite{kapouleas5,kapouleas2,kapouleas3}.  The idea is that a very good approximation of a CMC hypersurface can be constructed by forming the connected sum of an initial configuration of simple CMC building blocks, which can then be perturbed to an exactly CMC hypersurface if certain global geometric conditions, called \emph{balancing conditions}, are satisfied by the initial configuration.  

The gluing technique has been a very successful method for constructing CMC hypersurfaces in $\R^3$, with the proviso that the resulting hypersurfaces are always small perturbations of the simple building blocks from which they are constructed, namely spheres and nearly singular truncated Delaunay surfaces.  This is because the quality of the approximate solution that one can construct improves as the approximate solution more and more closely resembles a union of mutually tangent spheres.  Although it is easy to imagine how to use the gluing technique in ambient manifolds other than $\R^3$, provided one has enough simple building blocks, it is not clear that the gluing technique will be quite a successful, in particular when the ambient manifold is compact.

In Butscher's and Butscher-Pacard's work \cite{mepacard1,mepacard2,me5}, the gluing technique for constructing CMC hypersurfaces has been successfully adapted to work in the compact ambient manifold $\Sph^{n+1}$.   In these papers, the CMC building blocks of the sphere --- namely the hyperspheres obtained by intersecting $\Sph^{n+1}$ with hyperplanes and the product spheres of the form $\Sph^p(\cos(\alpha)) \times \Sph^q( \sin(\alpha))$ for $\alpha \in (0, \pi/2)$ called the generalized Clifford tori --- are configured in a variety of different ways, glued together using small embedded catenoidal necks, and perturbed into CMC hypersurfaces.  One should imagine that the hypersurfaces constructed in these papers are analogues of the `classical' constructions that are possible in Euclidean space.  As before, there are obstructions for solving the CMC equation on an arbitrary initial configuration; and when certain global geometric conditions are met, the obstructions disappear.  These geometric conditions are indeed close analogues of the balancing conditions identified by Kapouleas; but the conditions seem to be stronger in the $\Sph^{n+1}$ case than in the Euclidean case.  This is to be expected since $\Sph^{n+1}$ is compact and the additional requirement that the initial configurations must close should have ramifications in the analysis of the CMC equation.

The balancing condition is best explained in the more general context found in Korevaar-Kusner-Solomon's work \cite{kks}.   First, suppose that $\Lambda$ is a hypersurface with constant mean curvature $h$ in an $(n+1)$-dimensional Riemannian manifold $(M, g)$ possessing a Killing field $V$.  Let $\mathcal U$ be an open set in $\Lambda$ and $\bar{\mathcal U}$ be an open set in $M$ such that $\partial \bar{ \mathcal U} = \partial \mathcal U \cup C$ where $C$ is a bounded $n$-dimensional cap which may have multiple components.  Then the first variation formula for the $n$-volume of $\mathcal U$ subject to the constraint of constant enclosed $(n+1)$-volume of $\bar{\mathcal U}$ in the direction of the variation determined by $V$ implies 
\begin{equation}
	\mylabel{eqn:kksbalancing}
	\int_{\partial \mathcal U} g(\nu, V) - h \int_C g(\eta, V) = 0
\end{equation}
where $\nu$ is the unit normal vector field of $\partial \mathcal U$ in $\Lambda$ and $\eta$ is the unit normal vector field of $C$ in $M$.  This formula can now applied to the approximate solution of the CMC perturbation problem, having mean curvature approximately equal to $h$, in the following way.  Choose the open set $\mathcal U$ as one of the building blocks of the approximate solution.  Then $\partial \mathcal U$ consists of a disjoint union of small $(n-1)$-spheres at the centres of the necks attaching $\mathcal U$ to its neighbours, and $C$ is the disjoint union of the small disks that cap these spheres off.  The left hand side of \eqref{eqn:kksbalancing} now encodes information about the width and location of the neck regions of $\mathcal U$.  If the left hand side of \eqref{eqn:kksbalancing} is sufficiently close to zero, then one says that $\mathcal U$ is in \emph{balanced position}.  The idea is now that in order to be able to overcome the obstructions to the solvability of the constant mean curvature equations,  the approximate solution must be constructed in such a way that all its building blocks are in balanced position.

The balancing condition amounts to a form of local symmetry satisfied by each building block with respect to its nearest neighbours in the initial configuration that is to be glued together.  This is similar to what happens in Euclidean space.   However, force balancing in itself is not the end of the story --- a balanced approximate solution can not necessarily be perturbed to an exactly CMC hypersurface.  It is in addition necessary to be able to re-position the various building blocks with respect to each other so as to maintain the force balancing condition even under small perturbations.  Technically speaking, this amounts to the requirement that the mapping taking a re-positioned approximate solution to a set of small real numbers via the integrals on the left hand side of \eqref{eqn:kksbalancing} be surjective.  This requirement also exists in the Euclidean case, but is more restrictive in the case of $\Sph^{n+1}$.   In fact, only by imposing a high degree of symmetry on their initial configurations are Butscher and Butscher-Pacard able to satisfy both types of obstruction to the solvability of the CMC equation.

One impression that the reader might have, after studying the implementation of the gluing technique in $\Sph^{n+1}$ presented in Butscher and Butscher-Pacard's papers, is that it might not possible to construct CMC hypersurfaces in $\Sph^{n+1}$ that are not very symmetric.  Indeed, the totality of local symmetry conditions imposed by force balancing and the fact that CMC hypersurfaces in $\Sph^{n+1}$ must close seems to force a degree of global symmetry on the initial configuration; and the methods developed in \cite{me5} do not seem to apply perfectly to initial configurations with small symmetry groups.  This situation, if it were true, would be in contrast to the Euclidean case. 

However, the impression outlined above is false.  The purpose of this present paper is twofold: to show that the balancing technique applied to spherical building blocks, as presented in \cite{me5}, can be generalized to handle initial configurations with lesser symmetry; and that there exist initial configurations of hyperspheres to which this technique can be applied.  Nevertheless, it remains the case that the conditions guaranteeing the existence of CMC hypersurfaces constructed from spherical building blocks are more restrictive than in the Euclidean case, and constructions which are possible in Euclidean space are impossible to achieve using the gluing technique in $\Sph^{n+1}$.  Examples will be presented in Section \ref{sec:examples} to demonstrate both of these facts.

\paragraph{Statement of results.} 
The theorem that will be proved in this paper can be explained as follows.  Let $\Gamma := \{ \gamma_1, \ldots, \gamma_L \}$ be a set of oriented geodesic segments with the property that the one-dimensional variety $\bigcup_s \gamma_s$ has no boundary.  Without loss of generality: the points of contact between any two segments are always amongst the endpoints of the geodesics; and two segments are never parallel whenever they meet.  Thus the endpoints of each geodesic segment $\gamma_s$ make contact with at least two other segments.  Let $\{ p_1, \ldots, p_M \}$ be the set of all endpoints of the geodesic segments and for each $p_s$ let $T_{1, s}, \ldots, T_{N_s, s} \in T_{p_s} \Sph^{n+1}$ be the unit tangent vectors of the geodesics emanating from $p_s$.  Now position hyperspheres of radius $\cos(\alpha)$ separated by a distance $\tau_s$ along each of the geodesics, perhaps winding multiple times around $\Sph^{n+1}$.  Note that there is a transcendental relationship between the $\tau_s$ and the number of windings around $\gamma_s$ that must be satisfied for this to be possible.  Denote this initial configuration of hyperspheres by $\Lambda_ {\Gamma,\tau} ^\#$.  

In  Section \ref{sec:assembly} a procedure will be developed for gluing the hyperspheres in $\Lambda_{\Gamma,\tau}^\#$ together by embedding small catenoidal necks between each pair of hyperspheres to form an \emph{approximate solution}, denoted $\tilde \Lambda_ {\Gamma,\tau}$, of the CMC deformation problem.  It will be shown in Section \ref{sec:balancing} that the various hyperspheres in $\Lambda_ {\Gamma,\tau} ^\#$ are in balanced position if
\begin{equation}
	\label{eqn:balcond}
	\sum_{j=1}^{N_s} \eps_{j,s}^{n-1} T_{j,s} = 0
\end{equation}
for each point $p_s$, where $\eps_{j,s}$ is a parameter related to the separation parameter $\tau_{j,s}$ along the geodesic whose tangent vector is $T_{j,s}$.  (Actually, $\eps_{j,s}$ is the width of the neck connecting the hypersphere at $p_s$ to its neighbour in the direction of $T_{j,s}$.  The relation with $\tau_{j,s}$ will be established during the description of the gluing process).

\begin{mainthm} 
	\label{result1}
	Let $\Lambda_{\Gamma,\tau}^\#$ be the initial configuration of hyperspheres described above.  Suppose that balancing condition \eqref{eqn:balcond} holds and also that the mapping between finite-dimensional vector spaces which takes small displacements of the geodesics forming $\Lambda_{\Gamma,\tau}^\#$ to the quantity given by the left hand side of \eqref{eqn:balcond} has full rank.  If $\tau$ is sufficiently small, then $\tilde \Lambda_{\Gamma,\tau}$ can be perturbed into an exactly CMC hypersurface $\Lambda_{\Gamma, \tau}$.  This hypersurface can be described as a normal graph over $\tilde \Lambda_{\Gamma, \tau}$ where the graphing function has small $C^{2,\beta}$-norm.  In particular, $\Lambda_{\Gamma, \tau}$ is embedded if and only if $\tilde \Lambda_{\Gamma, \tau}$ is embedded.
\end{mainthm}

The proof of this theorem will follow broadly the same lines as Main Theorem 2 in Butscher's paper \cite{me5}.  That is, it will be shown that the partial differential equation for the graphing function  whose solution gives a CMC perturbation of $\tilde \Lambda_{\Gamma, \tau}$ can be solved up to a error term belonging to a finite dimensional obstruction space spanned by the approximate Jacobi fields of $\tilde \Lambda_{\Gamma, \tau}$ (as explained more fully in \cite{me5} and in the proof below).  Then it will be shown that the balancing conditions given in the theorem above are sufficient to eliminate the error term.

\paragraph{Acknowledgements.}
I would like to thank Frank Pacard for suggesting this problem to me, providing invaluable guidance to me during its completion, and showing me excellent hospitality during my visits to Paris.  I would also like to thank Rob Kusner, Rafe Mazzeo, Jesse Ratzkin and Rick Schoen for their support and assistance.

\section{Construction of the Approximate Solution}
\label{sec:assembly}

\subsection{The Initial Configuration of Hyperspheres}

Write $\R^{n+2}$ as $\R \times \R^{n+1}$ and give it the coordinates $(x^0, x^1, \ldots, x^{n+1})$.  Consider the hypersphere $$S_\alpha := \{ x \in \R^{n+2} \: : \: x^0 = \cos{\alpha} \:\: \mbox{and} \:\: (x^1)^2 + \cdots + (x^{n+1})^2 = \sin^2(\alpha) \} \, .$$ This hypersphere has constant mean curvature $H_\alpha$.   An arbitrary configuration of rotated copies of $S_\alpha$ positioned along geodesic segments can be defined concretely as follows.  

First let $\Gamma := \{ \gamma_1, \ldots, \gamma_L \}$ be a set of oriented geodesic segments with the property that the one-dimensional variety $\bigcup_s \gamma_s$ has no boundary.  Without loss of generality: the points of contact between any two segments are always amongst the endpoints of the geodesics; and two segments are never parallel whenever they meet.  Thus the endpoints of each geodesic segment $\gamma_s$ make contact with at least two other segments.   Let $|\gamma_s|$ be the length of $\gamma_s$ and use $\gamma_s (t)$ to denote the point on $\gamma_s$ lying a distance $t$ from its starting point.  Hence $t \longmapsto \gamma_s(t)$ is the arc length parametrization of $\gamma_s$.   Suppose that there is one fixed $\alpha \in (0,\pi/2)$ along with positive integers $N_s$ and $m_s$ and small separation parameters $\tau_s >0$ so that $|\gamma_s| + 2 \pi m_s = N_s (2 \alpha + \tau_s)$ for each $s = 1, \ldots, L$.   

Define the points $\mathring p_{sk} := \gamma_s ( k(2 \alpha + \tau_s) ) $ as well as the hyperspheres $\mathring{S}_{\alpha}^{sk} := \partial B_\alpha(\mathring p_{sk})$.  Thus the $\mathring{S}_{\alpha}^{sk} $ for $k = 0, \ldots, N_s$ are a collection of $N_s$  hyperspheres of the same mean curvature winding around the geodesic $\gamma_s$ a number $m_s$ times and separated from each other by a distance $\tau_s$.  The proof of the Main Theorem will in addition  require small displacements of the hyperspheres above from these `equilibrium' positions.  To this end, introduce the small displacement parameters $\vec \sigma_{sk} \in T_{\mathring p_{sk}} \Sph^{n+1}$.  Now define the points $p_{sk} := \exp_{\mathring p_{sk}} (\vec \sigma_{sk})$ as well as the hyperspheres $S_{\alpha}^{sk}[\vec \sigma_{sk}] := \partial B_{\alpha}(p_{sk})$.  To avoid ambiguity, the displacement parameter for any hypersphere corresponding to an endpoint of a geodesics must be unique; this is achieved by setting the appropriate $\vec \sigma_{s0}$ and $\vec \sigma_{s' N_{s'}}$ equal.   One can now define the initial configuration as follows. 

\begin{defn}
	The \emph{initial configuration} of  hyperspheres of mean curvature $H_\alpha$ positioned along the collection of geodesics $\Gamma$ having separation parameters $\tau := \{ \tau_1, \ldots, \tau_L \}$ and displacement parameters $\vec \sigma := \{ \vec \sigma_{10}, \ldots, \vec\sigma_{L N_L}\}$ is defined to be
	$$\Lambda^{\#} [\alpha, \Gamma, \tau, \vec\sigma] := \bigcup_{s=1}^L \bigcup_{k=0}^{N_s} S_\alpha^{sk} [\vec \sigma_{sk} ] \, .$$
	The initial configuration is symmetric with respect to the group $G_\Gamma$. Note that there is redundancy in the labeling above due to the intersections amongst the geodesics at their endpoints.
\end{defn}

Each constituent hypersphere $S^{sk}_\alpha$ in the $\Lambda^{\#}[\alpha, \Gamma, \tau, \vec \sigma]$ has at least two nearest neighbours.  If $k \neq 0, N_s$ then $S^{sk}_\alpha$ is situated near an interior point of the geodesic $\gamma_s$ and thus has exactly two nearest neighbours $S^{s,k-1}_\alpha$ and $S^{s,k+1}_\alpha$ along this geodesic.  If $k = 0$ or $N_s$ then $S^{sk}_\alpha$ is situated near an endpoint of the geodesic $\gamma_s$ and has strictly greater than two nearest neighbours corresponding to hyperspheres of the form $S^{s' k'}_\alpha$ where $s' \in \{0, \ldots, L \} \setminus \{ s \}$ and $k' = 1$ or $N_{s'}-1$.

Finally, one can choose once and for all an $SO(n+2)$-rotation $R_{sk}[\vec \sigma_{sk}]$ taking $S_\alpha^{sk}[\vec \sigma_{sk}]$ to $S_\alpha$ as follows.  First fix a particular $R_{sk} \in SO(n+2)$ take $\mathring S_{\alpha}^{sk}$ to $S_\alpha$ (here, the choice does not matter so long as it is fixed \emph{a priori}).  Then let $\mathcal W_{\vec \sigma_{sk}}$ be the distance-one rotation in the one-parameter family of rotations generated by the $(n+2) \times (n+2)$ anti-symmetric linear transformation given by $W_{\vec \sigma_{sk}} (X) := \langle \vec \sigma_{sk}, X \rangle \mathring p_{sk} - \langle \mathring p_{sk}, X \rangle \vec \sigma_{sk}$ for $X \in \R^{n+2}$.  This is the unique $SO(n+2)$-rotation that coincides with $\exp_{\mathring p_{sk}} (\vec \sigma_{sk})$ at $\mathring p_{sk}$.  Now define
\begin{equation}
	\label{eqn:rotation}
	R_{sk}[\vec \sigma_{sk}] := R_{sk} \circ \mathcal W_{\vec \sigma_{sk}}^{-1} \, .
\end{equation} 
A consequence is that the dependence of $R_{sk}[\vec \sigma_{sk}]$ on $\vec \sigma_{sk}$ is smooth.

\subsection{Symmetries}

Let $G_\Gamma$ be the largest subgroup of $O(n+2)$ preserving the collection of geodesics $\Gamma$.   The idea is that $G_\Gamma$ should become the group of symmetries of the CMC hypersurface constructed in the proof of the Main Theorem.  Therefore in all steps leading up to the proof of the Main Theorem, it will be necessary to ensure that invariance with respect to $G_\Gamma$ is preserved.  

The initial configuration $\Lambda^\# [ \alpha, \Gamma, \tau, 0]$ is clearly invariant with respect to $G_\Gamma$ but once non-zero displacement parameters are introduced, this may no longer be so.  To preserve $G_\Gamma$-invariance, it will be necessary to choose only special values of the displacement parameters.  Let $N := \sum_{s=1}^L (N_s+1)$ be the total number of hyperspheres in $\Lambda^\# [ \alpha, \Gamma, \tau, 0]$ so that here are a total of $N$ displacement parameters, each of which belongs to $\R^n$.  Define the set 
\begin{equation}
	\label{eqn:displace}
	\mathcal D_{\Gamma} := \left\{   \vec \sigma \in \R^{n} \times \stackrel{\mbox{\tiny $N$ times}}{\cdots\cdots} \times \R^n : \Lambda^{\#}[ \alpha, \Gamma, \tau, \vec \sigma] \: \: \mbox{is $G_\Gamma$-invariant}  \right\} \, .
\end{equation}
Henceforth the condition $\vec \sigma \in \mathcal D_{\Gamma}$ on the displacement parameters  will be assumed.

\subsection{Preliminary Perturbation of the Initial Configuration}
 
Let $\btheta : \Sph^n \rightarrow \R^{n+1}$ be a parametrization of the unit sphere in $\R^{n+1}$.  Then one can parametrize the hypersphere $S_\alpha$ via $\btheta \longmapsto (\cos(\alpha) , \sin(\alpha) \btheta)$.  Furthermore, the displacement by a distance $\sigma$ along the geodesic normal to a point on $S_\alpha$ is found using the exponential map and is given by 
$$
\exp(\sigma N_\alpha)(\btheta) = (\cos(\alpha + \sigma), \sin( \alpha + \sigma) \btheta)
$$
where $N_\alpha$ is the unit outward normal of $S_\alpha$.  Suppose now that $G: \Sph^n \rightarrow \R$ is a function on $\Sph^n$.  Then one can parametrize the \emph{normal graph} over $S_\alpha$ corresponding to $G$ via
$$
\btheta \longmapsto \big( \cos(\alpha + G(\btheta)), \sin(\alpha + G(\btheta)) \btheta \big) \
$$
where $\btheta$ ranges over $\Sph^n$. 

The idea of the first step of the construction of the approximate solution of the CMC deformation problem is to replace each hypersphere $S_\alpha^{sk}[\vec \sigma_{sk}]$ of the initial configuration $\Lambda^{\#}[\alpha, \Gamma, \tau, \vec \sigma]$ by a small perturbation.  This perturbation will be given as the normal graph over $S_\alpha^{sk}[\vec \sigma_{sk}]$ generated by a specific function $G_{sk} : S_\alpha^{sk}[\vec \sigma_{sk}] \rightarrow \R$.  The purpose of this initial perturbation is to give $S_\alpha^{sk}[\vec \sigma_{sk}]$ a catenoidal shape near its gluing points.

To proceed, recall that $S_\alpha^{sk}[\vec \sigma_{sk}] = R_{sk} [\vec \sigma_{sk}]^{-1} (S_\alpha)$.   Let $p_1, \ldots p_K$ be the images under $R_{sk}[\vec \sigma_{sk}]$ of the points on $S_\alpha^{sk}[\vec \sigma_{sk}]$ that are nearest to the neighbours of $S_\alpha^{sk}[\vec \sigma_{sk}]$ amongst the hyperspheres of $\Lambda^{\#}[\alpha, \Gamma, \tau, \vec \sigma]$.  Introduce a small radius parameter $r$ to be determined later and define 
\begin{equation}
	\label{eqn:pertsph}
	\tilde S^{sk}_\alpha[\vec a_{sk}, \vec \sigma_{sk}] := \big( R_{sk}[\vec \sigma_{sk}] \big)^{-1} \circ \exp \big( G_{sk} N_\alpha \big) \left(S_\alpha \setminus \bigcup_{j=1}^K B_{r}(p_j) \right) \, ,
\end{equation}
where $G_{sk} : S_\alpha \setminus \{ p_1, \ldots, p_K\} \rightarrow \R$ is the function determined by the following procedure.  

Write $p_j := ( \cos(\alpha), \sin(\alpha) P_j)$ where $P_j$ are points on $\Sph^n \subseteq \R^{n+1}$.  Let $\mathcal L_{\Sph^n} := \Delta_{\Sph^n} + n$  be the linearized mean curvature operator of $S_\alpha$.  Recall that the smooth kernel of $\mathcal L_{\Sph^n}$ consists of the linear span of the restrictions of the coordinate functions $q^t := x^t \big|_{\Sph^n}$ for $t = 1, \ldots, n$.  Let $\delta (p_{j})$ be the Dirac $\delta$-mass centered at the point $p_{j}$.  Then for each  $\vec a_{sk} := (a_1, \ldots a_K) \in \R^K$, one can find a unique solution $G_{sk} : S_\alpha \setminus \{ p_1, \ldots, p_K\} \rightarrow \R$ of the distributional equation 
\begin{equation}
	\label{eqn:distribeqn}
	\mathcal L_{\Sph^n} (G_{sk} ) = \sum_{j=1}^K a_{j} \left( \delta (p_{j}) - \sum_{t=1}^n\lambda_j^t \, \, \chi \cdot q^t \right)
\end{equation}
that is $L^2$-orthogonal to the smooth kernel of $\mathcal L_{\Sph^n}$.  Here $\chi$ is a cut-off function vanishing in a neighbourhood of each of the $p_j$ that will be defined precisely later, and the $\lambda_j^t \in \R$ are coefficients designed to ensure that the right hand side of \eqref{eqn:distribeqn} is $L^2$-orthogonal to all the $q^t$, thereby guaranteeing the existence of the solution. One can also show that the following asymptotic expansion is valid:
\begin{equation}
	\label{eqn:distribeqnasym}
	G_{sk}  =
	\begin{cases}
		a_j (c_0 + \log \big( \mathrm{dist}(\cdot, p_j) \big) + \mathcal O \big( \mathrm{dist}(\cdot, p_j) \big) &\qquad n=2 \\[1ex]
		\dfrac{a_j}{ \mathrm{dist}(\cdot, p_j)^{n-2}} + \mathcal O \big( \mathrm{dist}(\cdot, p_j) ^{3-n} \big) &\qquad n\geq 3
	\end{cases}
\end{equation}
in a sufficiently small neighbourhood of $p_j$.  Here $c_0$ is a constant.  Refer to $\vec a_{sk}$ as the \emph{asymptotic parameters} of the perturbed hypersphere $\tilde S^{sk}_\alpha[\vec a_{sk}, \vec \sigma_{sk}] $.

\subsection{Canonical Coordinates for a Pair of Hyperspheres} 
\label{sec:stereo}

Let $S = \partial B_{\alpha}(p)$ and $S' = \partial B_\alpha(p')$ be any pair of neighbouring hyperspheres in $\Lambda^{\#}[\alpha, \Gamma, \tau, \vec \sigma]$ and suppose that the separation between them is given by a distance $\tau$.  Let $\gamma$ be the geodesic connecting $p$ to $p'$ and choose an $SO(n+2)$-rotation $R$ that takes $\gamma$ to the $\{ x^0, x^1\}$-equator and the points $p$ and $p'$ to $(\cos(\tau/2), - \sin(\tau/2), 0, \ldots, 0)$ and $(\cos(\tau/2), \sin(\tau/2), 0, \ldots, 0)$  respectively.

Canonical coordinates adapted to the hyperspheres $S$ and $S'$ can be defined as follows. Let $p^\flat$ be the midpoint of $\gamma$, so that $R$ takes $p^\flat$ to the point $(1, 0, \ldots, 0)$.  Next, let $K : \Sph^{n+1} \setminus \{ (-1, 0, \ldots, 0) \} \rightarrow \R^{n+1}$ denote the stereographic projection centered at $(1, 0, \ldots, 0)$ defined by
$$ K(x^0, x^1, \ldots, x^{n+1}) := \left( \frac{x^1}{1+x^0} , \cdots, \frac{x^{n+1}}{1+x^0} \right)\, .
$$
Then define the desired coordinates by transplanting this stereographic projection to $p^\flat$ by composing with $R$.  That is, the desired coordinate mapping is the inverse of $K \circ R: \Sph^{n+1} \setminus \{ -p^\flat \} \rightarrow \R^{n+1}$.   Give the target $\R^{n+1}$ the coordinates $y^1, \ldots, y^{n+1}$ and set $\hat y := (y^2, \ldots, y^{n+1})$.  

Recall that stereographic projection sends non-equatorial $k$-spheres in $\Sph^{n+1}$ to $k$-spheres in $\R^{n+1}$ and sends equatorial $k$-spheres to linear subspaces.  One thus expects that the coordinate images of $S$ and $S'$ are two hyperspheres symmetrically located on either side of the origin and centered at points on the  $y^1$-axis.  Indeed, one can check that any point $ (\cos(\alpha), \sin (\alpha) \cos(\mu), \sin(\alpha) \sin(\mu) \Theta) \in S_\alpha$ that is rotated along $\gamma$ by an angle of $\tau/2$ and then rotated into $S^{sk}_\alpha [\vec \sigma_{sk}]$ by $R^{-1}$ maps to the point $(y^1, \hat y)$ in $\R^{n+1}$ given by
\begin{equation}
	\mylabel{eqn:stereo}
	\begin{aligned}
		y^1 & = \frac{\pm \big( \sin(\alpha + \tau/2) \cos(\alpha ) - \cos(\alpha + \tau/2) \sin(\alpha ) \cos(\mu) \big) }{1 + \cos(\alpha + \tau/2) \cos(\alpha) + \sin(\alpha + \tau/2) \sin(\alpha) \cos(\mu)}\\[1ex]
		\hat y &=  \frac{\sin(\alpha ) \sin(\mu) \Theta}{1 + \cos(\alpha + \tau/2) \cos(\alpha) + \sin(\alpha + \tau/2) \sin(\alpha ) \cos(\mu)} 
	\end{aligned}
\end{equation}
from which one can check that $(y^1, \hat y )$  lies on the locus of points satisfying the equation
\begin{equation}
	\mylabel{eqn:sphimage}
	\left( y^1 \pm d \right)^2 + \Vert \hat y \Vert^2 =  r^2
\end{equation}
where 
\begin{equation}
	\mylabel{eqn:raddisp}
	\begin{aligned}
		r &= r(\alpha, \tau) := \frac{\sin(\alpha)}{ \cos(\alpha) + \cos(\alpha + \tau/2)} \\
		d &= d(\alpha, \tau) := \frac{ \sin(\alpha + \tau/2)}{\cos(\alpha) + \cos(\alpha + \tau /2)} \, .
	\end{aligned}
\end{equation} 

An additional by-product of the stereographic coordinates defined above is that these bring the metric into geodesic normal form: that is, the metric is Euclidean and its derivatives vanish at the centre of the coordinates.  This can be seen by the computation of the metric $$\big( R^{-1} \circ K^{-1} \big)^\ast g_{\Sph^{n+1}} = A^{-2} g_0$$ where $g_0$ is the Euclidean metric of $\R^{n+1}$ and $A(y) = \frac{1}{2} (1 + \sum_{k=1}^{n+1} (y^k)^2)$.   The geodesic normal form will have the effect of distorting as little as possible the geometry of objects embedded into the sphere using the stereographic coordinate map, provided one remains near the origin.

\subsection{Gluing a Pair of Perturbed Hyperspheres Together}
\label{sec:glue}

\paragraph*{Asymptotic expansions.}  Because of the asymptotic expansion \eqref{eqn:distribeqnasym}, the perturbed hyperspheres $\tilde S^{sk}_\alpha [\vec a_{sk}, \vec \sigma_{sk}]$ resemble the ends of catenoids near their boundaries.  By reparametrizing the images of the perturbed hyperspheres under stereographic projection as graphs over the $\hat y$-hyperplane in a small neighbourhood of the $y^1$-axis, one obtains a more precise description of this resemblance. 

Let $S$ and $S'$ be the pair of neighbouring hyperspheres studied above and suppose $R \in SO(n+2)$ carries them into the standard position with respect to the $\{x^0, x^1\}$-equator.  Suppose that $\tilde S$ and $\tilde S'$ are the corresponding perturbed hyperspheres generated by the functions $G$ and $G' $.  Set $a := a_1$ and $a' := a_1'$ in the definition of $G$ and $G'$.  Furthermore, suppose that $p_1$ and $p_1'$ are the points of closest approach between $S$ and $S'$, and that these are separated by a distance $\tau$.

From the formul\ae\ for normal graphs over hyperspheres and the properties of the stereographic projection, one finds that the coordinates $y(\mu, \Theta) \in \R^{n+1}$ of a point on the image of the perturbed hypersphere $K \circ R(\tilde S)$ near the point $K \circ R(p_1)$ satisfy 
\begin{equation}
	\mylabel{eqn:vgraph}
	y^1(\mu) = -d_G(\mu) + \sqrt{[r_G(\mu)]^2 - \| \hat y \|^2}
\end{equation}
where
\begin{align*}
	r_G(\mu) &:=  \frac{\sin(\alpha+  G (\mu))}{ \cos(\alpha+ G (\mu)) + \cos(\alpha + \tau/2)} \\[0.5ex]
	d_G(\mu) &:=  \frac{ \sin(\alpha + \tau/2)}{\cos(\alpha+ G (\mu)) + \cos(\alpha + \tau /2)} \, .
\end{align*}
Furthermore, the relation between $\mu$ and $\| \hat y \|$ is given by 
\begin{equation}
	\mylabel{eqn:invrel}
	\| \hat y \| = \frac{\sin(\alpha+  G (\mu) ) \sin(\mu)}{1 + \cos(\alpha + \tau/2) \cos(\alpha+ G (\mu)) + \sin(\alpha + \tau/2) \sin(\alpha+ G (\mu) ) \cos(\mu)} \, .
\end{equation}
By computing the derivative $\frac{\dif}{\dif \mu} \| \hat y \|$, one finds that the relation \eqref{eqn:invrel} is invertible in the region where both $G(\mu)$ and $\mu$ are small, and moreover that $\mu(\| \hat y \|) = 2 \csc(\alpha) \cos^2(\tau/4) \| \hat y \| + \mathcal O(\| \hat y \|^3)$.  Substituting this into \eqref{eqn:vgraph} yields $y^1(\| \hat y \|) = \mathcal G(\| \hat y \|)$ where $\mathcal G(\| \hat y \|) := d_G( \mu(\| \hat y \|)) - \sqrt{[r_G(\mu(\| \hat y \|))]^2 - \| \hat y \|^2}$.  One finds also the asymptotic expansion 
\begin{equation}
	\mylabel{eqn:greensurfasym}
	\mathcal G(\| \hat y \|) =
	\begin{cases}
		\begin{aligned}
			&\hspace{-0.875ex} - \tan(\tau/4) - \frac{\| \hat y \|^2}{2r} + a \big( c_2 -  C_2 \log ( \| \hat y \|) \big) \\[-0.25ex]
			&\hspace{-0.875ex} \qquad \qquad + \mathcal O (\| \hat y \|^4) + \mathcal O(|a| \| \hat y \|^2 | \log(\| \hat y \|)| ) 
		\end{aligned} 
		&\quad \mbox{$n=2$}  \\[4ex]
		\displaystyle -\tan(\tau/4)  - \frac{\| \hat y \|^2}{2r}  + \frac{a C_3}{ \| \hat y \|} + \mathcal O (\| \hat y \|^4) + \mathcal O \left( |a| \| \hat y \| \right) &\quad \mbox{$n=3$}\\[1.5ex]
		\displaystyle - \tan(\tau/4)  - \frac{\| \hat y \|^2}{2r}   + \frac{a C_4}{ 2  \| \hat y \|^{2}} + \mathcal O (\| \hat y \|^4)  +\mathcal O \left( |a| \bigl[ 1 + | \log(\| \hat y \|) | \bigr]  \right) &\quad \mbox{$n=4$} \\[1.5ex]
		\displaystyle -\tan(\tau/4)  - \frac{\| \hat y \|^2}{2r}  + \frac{a C_n}{(n-2) \| \hat y \|^{n-2}} + \mathcal O (\| \hat y \|^4) + \mathcal O \left( \frac{|a|}{\| \hat y \|^{n-4}} \right) &\quad \mbox{$n\geq 5$}
	\end{cases}
\end{equation}
in the region where both $\| \hat y \|$ and $G(\mu(\| \hat y \|))$ remain small.  Here $d = d(\alpha, \tau)$ and $r = r(\alpha, \tau)$ are the quantities in \eqref{eqn:raddisp}, while 
\begin{equation}
	\label{eqn:expansionconst}
	\begin{aligned}
		c_2 &=  \frac{\cos(2 \alpha + \tau/2) \bigl(c_0  + \log \bigl( \csc(\alpha) \cos^2(\tau/4) \bigr) \bigr)}{\bigl( \cos(\alpha) + \cos(\alpha + \tau/2) \bigr)^2} \\[0.5ex]
		C_n &= \frac{\cos(2 \alpha + \tau/2)}{\bigl( \cos(\alpha) + \cos(\alpha + \tau/2) \bigr)^2 \bigl( 2 \csc(\alpha) \cos^2(\tau/4) \bigr)^{n-2}}
	\, .
	\end{aligned}
\end{equation}
In a similar manner, one finds that the equation satisfied by points on the image of the perturbed hypersphere $K \circ R(\tilde S')$ near the point $K \circ R(p_1')$ satisfy  $y^1(\| \hat y \|) = - \mathcal G'(\| \hat y \|)$ where $\mathcal G'(\| \hat y \|)$ is formally the same function as $\mathcal G(\| \hat y \|)$ except with $a$ replaced by $a'$.  

\paragraph*{Finding a matching catenoidal neck.} The catenoid in  $\R \times \R^n$ scaled by a factor $\eps>0$ is the hypersurface $\eps \Sigma $ parametrized by
	\begin{equation}
		\label{eqn:gencat}
		( s , \Theta) \in \RÊ\times \Sph^{n-1} \longmapsto (\eps  \psi (s), \eps \phi(s) \,  \Theta )
	\end{equation}
	where $\phi(s) :=  (\cosh (n-1) s)^{1/(n-1)}$ and $\psi (s) : = \int_{0}^s \, \phi^{2-n} (\sigma) \, \dif \sigma$, while $\Theta$ is parametrizes the unit $(n-1)$-sphere.  An alternate parametrization for the catenoid is needed here, namely when $\eps \Sigma$ is written as the union of two graphs over the $\R^n$ factor.  That is, by inverting the equation $\| \hat y \| = \eps \phi(s)$, one finds that  $\eps \Sigma = \Sigma_{\eps}^+ \cup \Sigma_{\eps}^-$ where $\Sigma_{\eps}^\pm := \big\{ ( F_\eps (\| \hat y \|), \hat y ) : \Vert \hat y  \Vert \geq \eps \big\}$.  The function $F_\eps : \{ x \in \R : x \geq \eps \} \rightarrow \R$ is defined by $F_\eps (x) =  \eps F(x /\eps)$ where 
\begin{equation}
	\mylabel{eqn:catenoidgrfn}
	F(x) := \int_1^x (\sigma^{2n-2} - 1)^{-1/2} \dif \sigma \, .
\end{equation}
Note that in dimension $n=2$ this function is simply $F(x) = \arccosh(x)$.  Therefore one has the asymptotic expansion
\begin{equation}
	\mylabel{eqn:catenoidasym}
	\eps F( \| \hat y \| / \eps) = 
	\begin{cases}\displaystyle
		\eps  \log (2 / \eps)  + \eps \log ( \| \hat y \| ) - \frac{\eps^3}{4 \| \hat y \|^2 } + \mathcal O \left( \frac{ \eps^5}{\| \hat y \|^4} \right) &\quad \mbox{$n=2$} \\[1.5ex]
		\displaystyle \eps c_n - \frac{\eps^{n-1}}{(n-2) \| \hat y \|^{n-2}} - \frac{\eps^{3n-3}}{2(3n-4) \| \hat y \|^{3n-4}} + \mathcal O \left( \frac{\eps^{5n-5}}{\| \hat y \|^{5n-6}} \right) &\quad \mbox{$n \geq 3$}
	\end{cases}
\end{equation}
where $c_n$ is yet another constant.  

In order to find the catenoid which matches optimally with the coordinate images $K \circ R(\tilde S)$ and $K \circ R(\tilde S')$, one must compare the asymptotic expansions of type \eqref{eqn:greensurfasym} valid for $K \circ R(\tilde S)$ and $K \circ R(\tilde S')$ with the asymptotic expansion \eqref{eqn:catenoidasym} at either end of the catenoid.  These asymptotic expansions match if the following conditions hold:
\begin{subequations}
	\label{eqn:match}
	\begin{equation}
		a C_n = \eps^{n-1} = a' C_n
	\end{equation}
	as well as
		\begin{equation}
		\begin{aligned}
			\left. 
			\begin{aligned}
				-\tan(\tau/4) +  a c_2 &=  - \eps \log(2/\eps) \\
				\tan(\tau/4) - a c_2 &=  \eps \log(2/\eps) 
			\end{aligned}
			\qquad \right\} &\quad n=2 \\
			\left.
			\begin{aligned}
				-\tan(\tau/4) &= - \eps c_n \\
				\tan(\tau/4)  &=   \eps c_n 
			\end{aligned}
			\hspace{1.825cm} \right\} &\quad n\geq 3 \, .
		\end{aligned}
	\end{equation}
\end{subequations}
It is clear that these equations determine $a, a'$ and $\eps$ completely in terms of the separation $\tau$ between the hyperspheres.   In fact $\eps = c_n^{-1} \tan (\tau/4)$ in dimension $n\geq 3$ and $\eps$ satisfies $\tan(\tau/4) = c_2 C_2^{-1} \eps  + \eps \log(2/\eps)$ in dimension $n=2$.

\paragraph*{The gluing.}  The considerations above determine the parameter values for the two perturbed hyperspheres and the re-scaled catenoid needed for optimal matching.  The height of the matching catenoid can also be determined by these considerations.  That is, once $a, a', \eps$ have been found in terms of $\tau$, then the errors $| \eps F(\| \hat y \|/\eps) - \mathcal G (\| \hat y \|)|$ and $| - \eps F(\| \hat y \|/\eps) + \mathcal G' (\| \hat y \|)|$ are smallest when $\hat y$ is chosen to lie in a range where the quantity $\frac{1}{2r} \| \hat y \|^2 + \frac{\eps^{3n-3}}{2(3n-4)} \| \hat y \|^{4-3n}$ is minimized.  This occurs when $\| \hat y \| = \mathcal O(r_\eps)$ where $r_\eps := \eps^{(3n-3)/(3n-2)}$.  Thus one must truncate the re-scaled catenoid exactly at $\| \hat y \| = r_\eps$ for an optimally smooth gluing.

The gluing itself can now be accomplished as follows.  Let $\eta : [0, \infty) \rightarrow \R$ be a smooth, monotone cut-off function satisfying $\eta(s) = 0$ for $s \in [0, 1/2]$ and $\eta(s) = 1$ for $s \in [2, \infty)$.  Define the functions $\tilde F_{\alpha, \tau}^\pm : \bar B_{2 r_\eps}(0) \setminus B_{r_\eps} (0) \subseteq  \R^{n} \rightarrow \R$ by 
\begin{equation}
	\mylabel{eqn:mergedfn}
	\begin{aligned}
		\tilde F^+_{\alpha, \tau} (\hat y) &= \eps \big( 1 - \eta(\Vert \hat y\Vert / r_\eps) \big)  F(\Vert \hat y\Vert / \eps)  + \eta(\Vert \hat y\Vert /r_\eps) \mathcal G (\| \hat y \| ) \\
		\tilde F^-_{\alpha, \tau} (\hat y) &= - \eps \big( 1 - \eta(\Vert \hat y\Vert / r_\eps) \big)  F(\Vert \hat y\Vert / \eps)  - \eta(\Vert \hat y\Vert /r_\eps) \mathcal G' (\| \hat y \| ) \, .
	\end{aligned}
\end{equation}
Define the hypersurfaces $\tilde \Sigma^\pm_\eps = \{ ( \pm \tilde F^\pm_{a, \tau}(\hat y), \hat y) : \Vert \hat y \Vert \in [\eps, 2 r_\eps] \}$.  Therefore $\tilde \Sigma_\eps := \tilde \Sigma_\eps^+ \cup \tilde \Sigma_\eps^-$ is a smooth hypersurface connecting $K \circ R( \tilde S) \setminus \big[  \R \times B_{2 r_\eps}(0) \big]$ to  $K \circ R( \tilde S') \setminus \big[ \R \times B_{2 r_\eps}(0) \big]$ through a truncated catenoid.

\subsection{Assembling the Approximate Solution}

All neighbouring hyperspheres in the initial configuration $\Lambda^{\#}[\alpha, \Gamma, \tau, \vec \sigma]$ can be perturbed appropriately and glued together repeating the process described in the previous three sections for every hypersphere.  That is, the separation between any two hyperspheres in $\Lambda^{\#}[\alpha, \Gamma, \tau, \vec \sigma]$ determines the parameters of the catenoidal neck that fits between them via the equations \eqref{eqn:match}.  The neck scale parameters of all the necks then determine the asymptotic parameters of the perturbed hyperspheres.  Finally, each perturbed hypersphere is attached to its neighbouring necks using the method of fusing the graphing functions of the neck with the graphing functions of the perturbed hyperspheres as in equation \eqref{eqn:mergedfn}.

\begin{defn}
\mylabel{defn:approxsol}
The \emph{approximate solution} with parameters $\tau$, $\vec \sigma$ is the hypersurface $\approxsol$ obtained from the process described above.
\end{defn}  

Note that by choosing the functions $G_{sk}$ invariant under all $\rho \in G_\Gamma$ preserving $S_{\alpha}^{sk}[\vec \sigma_{sk}]$ and equal to $G_{s'k'}$ whenever there is $\rho \in G_\Gamma$ taking $S_{\alpha}^{sk}[\vec \sigma_{sk}]$ to $S_{\alpha}^{s'k'}[\vec \sigma_{s'k'}]$, then $\approxsol$ becomes invariant under $G_{\Gamma}$ as well.  Finally,  the hypersurface $\approxsol$ can be subdivided into regions of three distinct types.
\begin{defn}
	Identify the following regions of $\approxsol$.
	\begin{itemize}
		\item Let $\mathcal N^{sk}$ be the \emph{neck region} between the $k^{\mathit{th}}$ and $(k+1)^{\mathit{st}}$ perturbed hypersphere along the geodesic $\gamma_s$.  Note that $\mathcal N^{sk}$ carries a scale parameter $\eps_{sk}$ depending smoothly on $\tau_s$ and $\vec \sigma_{sk}$ and $\vec \sigma_{s, k+1}$.   In the canonical stereographic coordinates $\mathcal N^{sk}$ is the set of points $(y^1, \hat y)$ corresponding to $\| \hat y \| \leq r_{\eps_{sk}}$.
		
		\item Let $\mathcal T^{sk, \pm}$ be the \emph{transition regions} associated to the neck $\mathcal N^{sk}$.   In the canonical stereographic coordinates used to define this neck, $\mathcal T^{sk, +}$ is the set of points $(y^1, \hat y)$ corresponding to $r_{\eps_{sk}} < \| \hat y \| \leq2 r_{\eps_{sk}}$ and $y^1 > 0$ whereas $\mathcal T^{sk, -}$ is the set of points $(y^1, \hat y)$ corresponding to $r_{\eps_{sk}} < \| \hat y \| \leq2 r_{\eps_{sk}}$ and $y^1 < 0$.

		\item Let $\mathcal E^{sk}$ be the \emph{spherical region} corresponding to the $k^{\mathit{th}}$ neck along the geodesic $\gamma_s$.  This is the set of points in $\tilde S^{sk}_\alpha [ \vec a_{sk}, \vec \sigma_{sk}] \setminus \bigcup_{j=1}^K B_{r} (p_j)$ where $p_1, \ldots, p_K$ are the points of $S^{sk}_\alpha [\vec \sigma_{sk}]$ closest to its neighbouring hyperspheres and $r$ is a small radius chosen to exclude exactly the neck and transition region connecting $\tilde S^{sk}_\alpha [ \vec a_{sk}, \vec \sigma_{sk}]$ to its neighbour near $p_j$.
	\end{itemize}
\end{defn}

\section{The Exact Solution up to Finite-Dimensional Error}
 
\subsection{The Analytic Set-Up}

\paragraph*{Deforming the approximate solution.} The approximate solution $\approxsol$ has mean curvature almost equal to $H_\alpha$ everywhere except in the neck and transition regions where the mean curvature becomes zero.  The task ahead is to develop a means for deforming $\approxsol$ as well as an equation that selects the deformation of $\approxsol$ into an exactly constant mean curvature hypersurface with mean curvature equal to $H_\alpha$.  

Since $\approxsol$ is a hypersurface, it is possible to parametrize deformations of $\approxsol$ in a very standard way via normal deformations.  These can be constructed by choosing a function $f : \approxsol \rightarrow \R$ and then considering the deformation $\phi_f : \approxsol \rightarrow \Sph^{n+1}$ given by $\phi_f (x) := \exp_x (f(x) \cdot N(x))$ where $\exp_x$ is the exponential map at the point $x$ and $N(x)$ is the outward unit normal vector field of $\approxsol$ at the point $x$.  For any given function $f$, the hypersurface $\phi_f(\approxsol)$ is a normal graph over $\approxsol$, provided $f$ is sufficiently small in a $C^1$ sense.  Finding an exactly CMC normal graph near $\approxsol$ therefore consists of finding a function $f$ satisfying the equation $H_{\phi_f(\approxsol)} = H_\alpha$, where $H_\Lambda$ denotes the mean curvature of a hypersurface $\Lambda$.  

\begin{defn} 
Let  $\difop$ be the operator $f \longmapsto H_{\phi_f(\approxsol)} - H_\alpha $.  
\end{defn}

\noindent This is a quasi-linear, second-order partial differential operator for the function $f$ whose zero gives the desired deformation of $\approxsol$.  

\paragraph*{The strategy of the proof.} Finding a solution of the equation $\difop(f) = 0$ when $\tau$  and $\vec \sigma$ are sufficiently small will be accomplished by invoking the \emph{Banach space inverse function theorem} in a particular way.   To provide a focus for the remainder of the proof of Main Theorem \ref{result1}, this fundamental result will be stated here in fairly general terms \cite{amr}.

\begin{nonumthm}[IFT]
	Let $\Phi : X \rightarrow Z$ be a smooth map of Banach spaces, set $\Phi(0) := E$ and define the linearized operator $\mathcal L := \Dif \Phi(f) =  \left. \frac{\dif}{\dif s} \Phi(f+ s u)  \right|_{s=0}$.  Suppose $\mathcal L $ is bounded and surjective, possessing a bounded right inverse $\mathcal R : Z \rightarrow X$ satisfying
	\begin{equation}
		\mylabel{eqn:iftestone}
		\Vert \mathcal R (z)\Vert \leq C \Vert z \Vert
	\end{equation}
	for all $z \in Z$.	 Choose $R$ so that if $y \in B_R(0) \subseteq X$, then 
	\begin{equation}
		\mylabel{eqn:iftesttwo}
		\Vert \mathcal L( x) - \Dif \Phi(y) (x) \Vert \leq \frac{1}{2C}  \Vert x \Vert
	\end{equation}
for all $x \in X$, where $C>0$ is a constant.  Then if $z \in Z$ is such that 
	\begin{equation}
		\mylabel{eqn:iftestthree}
		\Vert z - E \Vert \leq \frac{R}{2C} \, ,
	\end{equation}
there exists a unique $x \in B_R(0)$ so that $\Phi(x) = z$.  Moreover, $\Vert x \Vert \leq 2 C \Vert z - E \Vert$. 
\end{nonumthm}

As the statement of theorem makes clear, it must be the case that $\linop$ is surjective with a bounded right inverse in order to solve the equation $\difop(f) = 0$ in a Banach subspace $X$ of at least $C^2$ functions on $\approxsol$.   It is however a general phenomenon in singular perturbation problems that the linearized operator often has a finite number of small eigenvalues tending to zero as the singular parameter (in this case $\tau$) tends to zero, making it impossible to achieve the bound \eqref{eqn:iftestone} with a constant independent of $\tau$.  The eigenfunctions associated to these degenerating eigenvalues are the \emph{obstructions} preventing the deformation to an exactly CMC hypersurface and are called \emph{Jacobi fields}.

One way out of this difficulty is to project $\linop$ onto a subspace of functions which is transverse to the co-kernel associated to the Jacobi fields, at least in a good enough approximate sense, and to construct a bounded right inverse for the projected linear operator.  Since  $\linop$ is self-adjoint, an appropriate subspace to choose is the $L^2$-orthogonal complement of the span of a collection of functions that closely approximates the Jacobi fields.  Let $\pi$ denote this $L^2$ projection (to be defined more precisely below) and set $\linop^\perp := \pi \circ \linop$.  The projected non-linear operator $\difop^\perp := \pi \circ \difop$, whose linearization is $\linop^\perp := \pi \circ \linop$, would then satisfy the requirements of the Banach space inverse function theorem and a solution of the equation $\difop^\perp (f) = 0$ could be found. This is the solution of the CMC deformation problem \emph{up to a finite-dimensional error term} lying in the span of the approximate Jacobi fields.  

The present section of this paper implements the above idea.  The construction of the right inverse satisfying the first of the estimates \ref{eqn:iftestone} needed to invoke the Banach space inverse function theorem (hereinafter called the \emph{linear estimate}) is carried out in Section \ref{sec:linest} after some preliminary work in Sections \ref{sec:jacobi} and \ref{sec:funcspace} that helps identify the correct Banach subspace $X$ and $L^2$ projection $\pi$.  The remaining two estimates \ref{eqn:iftesttwo} and \ref{eqn:iftestthree} required to invoke this theorem (hereinafter called the \emph{non-linear estimates}) are proved in Section \ref{subsec:nonlinest}.    Of course, it remains to show that it is possible to solve $\difop(f) = 0$ exactly, and this will be explained fully in Section \ref{sec:balancing}.

\subsection{Function Spaces and Norms}
\mylabel{sec:funcspace}

The equation $\pi \circ \difop (f) = 0$ will be solved in a Banach subspace of $C^{2,\beta}(\approxsol)$ where the norm is designed to properly determine the dependence on the parameter $\tau$ of the various estimates needed for the application of the inverse function theorem.  The norm in question is a so-called \emph{weighted Schauder norm}.  To define this norm, one must first define a \emph{weight function} on $\approxsol$.   Let $\mathcal P := \{ p_{sk}^\flat : k=0, \ldots, N_s - 1 \mbox{ and }  s = 1, \ldots, L \}$ be the set of all points of $\Sph^{n+1}$ upon which the necks of $\approxsol$ are centered.  Let $K_{sk}$ denote the canonical stereographic projection used to define the neck $\mathcal N^{sk}$.  Fix some $r_0$ independent of $\tau$ such that the balls of radii $2 r_0$ centered on any two points of $\mathcal P$ do not intersect.

\begin{defn}  
	\mylabel{defn:weight}
	The \emph{weight function} $\zeta_{\tau} : \approxsol \rightarrow \R$ is defined by
\begin{equation*}
	\zeta_\tau (x) = 
	\begin{cases}
		\eps_{sk} \cosh(s) &\quad x = K_{sk}^{-1} ( \eps_{sk} \psi(s), \eps_{sk} \phi(s) \Theta) \in \mathcal N^{sk} \\
		\mathit{Interpolation} &\quad x \in \mathcal T^{sk} \\
		\sqrt{ \eps_{sk}^2 + \mathrm{dist} (x, p_{sk}^\flat)^2} &\quad x \in \approxsol \cap \big[ \bar B_{r_0}(p_{sk}^\flat) \setminus \mathcal T^{sk,-} \cap \mathcal N^{sk} \cup \mathcal T^{sk,+} \big]\\[1ex]
		\mathit{Interpolation} &\quad x \in \approxsol \cap \big[ \bar B_{2 r_0}(p_{sk}^\flat) \setminus B_{r_0}(p_{sk}^\flat) \big]\\[1ex]
		2 r_0 &\quad x \in \approxsol \setminus \displaystyle \bigcup_{\mathcal P} B_{2 r_0}(p_{sk}^\flat) \, .
	\end{cases}
\end{equation*}	
\end{defn}

\noindent The interpolation is such that $\zeta_\tau$ is smooth and monotone in the region of interpolation, and invariant under the group $G_\Gamma$.  

The weighted Schauder norm can now be defined.   Let $\mathcal E := \approxsol \setminus \bigcup_{\mathcal P} B_{2 r_0}(p_{sk}^\flat)$ and let $\mathcal A^{sk}_r := \approxsol \cap \big[ B_{2 r}(p_{sk}^\flat) \setminus \bar B_r (p_{sk}^\flat) \big]$ for each $p_{sk}^\flat \in \mathcal P$.  Introduce the following preliminary notation.  For any tensor field $T$  on $\approxsol$ and any open subset $\mathcal U \subseteq \approxsol$, define 
\begin{equation*}
	| T |_{0, \mathcal U} := \sup_{x \in \,  \mathcal U} \Vert T(x) \Vert \\
	\qquad \mbox{and} \qquad [T]_{\beta, \, \mathcal U} := \sup_{x,x' \in \, \mathcal U} \frac{\Vert T(x') - \Xi_{x,x'} (T(x)) \Vert}{\mathrm{dist}(x,x')^\beta} \, ,
\end{equation*}
where the norms and the distance function that appear are taken with respect to the induced metric of $\approxsol$, while $\Xi_{x,x'}$ is the parallel transport operator from $x$ to $x'$.  Furthermore, if $f: \mathcal U \rightarrow \R$ then define
$$|f|_{l, \beta, \, \mathcal U} := \sum_{i=0}^{l}  | \nabla^i f |_{0, \, \mathcal U} + [ \nabla^l f ]_{\beta, \,\mathcal U} \, .$$
Now make the following definition.

\begin{defn} 
	Let $\mathcal U \subseteq \approxsol$ and $\delta \in \R$ and $\beta \in (0,1)$. The $C^{l, \beta}_\delta$ norm of a function defined on $\mathcal U$ is given by
\begin{equation}
	\mylabel{eqn:weightnorm}
	| f |_{C^{l, \beta}_\delta (\mathcal U)} :=  |  f |_{l, \beta, \,\mathcal U \cap \mathcal E} 
	+ \sup_{\mathcal P} \sup_{r \in [0,  r_0 ] } \left\{ \left( \sup_{x \in \mathcal U \cap \mathcal A_r^{sk}} \big[ \zeta_\tau (x) \big]^{-\delta} \right)  | f |_{l, \beta, \delta, \, \mathcal U \cap \mathcal A_r^{sk} } \right\} \, . 
\end{equation}  
\end{defn}

\noindent The notation for the $C^{l,\beta}_\delta$ norm of a function defined on all of  $\approxsol$ will be abbreviated $| \cdot |_{C^{l, \beta}_\delta}$ but $C^{l,\beta}_\delta$ norms of functions defined on smaller subsets of $\approxsol$ will be written out in full.  Finally, the Banach space $\clbg (\approxsol)$ denotes the $C^{l,\beta}$ functions of  $\approxsol$ measured with respect to the norm \eqref{eqn:weightnorm}, while $C^{l, \beta}_{\delta, \mathit{sym}}(\approxsol)$ denotes those functions $f \in \clbg (\approxsol)$ satisfying $f \circ \rho = f$ for all $\rho \in G_\Gamma$.

It is well known that all the `usual' properties that one would expect from a Schauder norm (multiplicative properties, interpolation inequalities, and so on) are satisfied by the weighted $\clbg$ norms.   It is thus easy to deduce that 
$$\difop : \ctbg(\approxsol) \rightarrow \cobg(\approxsol)$$
is a well-defined and smooth operator (in the sense of Banach spaces) and that %
$$\linop : \ctbg(\approxsol) \rightarrow \cobg(\approxsol)$$
is bounded in the operator norm by a constant independent of $\tau$.  Furthermore $\difop$ and $\linop$ can by symmetrized to yield  new operators (which will be given the same names) on the symmetrized $\clbg$ spaces.

\subsection{Jacobi Fields}  

\mylabel{sec:jacobi}

The obstructions preventing the solvability of the CMC deformation problem have a geometric origin.   To see this, recall the general fact that any one-parameter family of isometries of the ambient space in which a CMC hypersurface is situated gives rise to an element in the kernel of the linearized mean curvature operator as follows.

\begin{lemma}
	\mylabel{lemma:linmeancurv}
	Let $\Lambda$ be a closed hypersurface in a Riemannian manifold $X$ with mean curvature $H_\Lambda$, second fundamental form $B_\Lambda$ and unit normal vector field $N_\Lambda$.   If $R_t$ is a one-parameter family of isometries of $X$ with deformation vector field $V = \left. \frac{\dif }{\dif t} R_t \right|_{t=0} $, then the function $q_V := \langle V, N_\Lambda \rangle$ belongs to the kernel of $\Lambda$.
\end{lemma}

\begin{proof}
	Since $R_t$ is a family of isometries, then $H(R_t(\Lambda) ) = H(\Lambda)$ for all $t$ and $\left. \frac{\dif}{\dif t} \right|_{t=0}H(R_t(\Lambda)) = 0$.  The function $q_V := \langle V, N_\Lambda\rangle$ is thus in the kernel of $\Dif H_\Lambda(0)$ because $q_V$ generates a normal deformation of $\Lambda$ whose action, to first order, coincides with $R_t$.
\end{proof}

The obstructions preventing the solvability of the equation $\difop(f) = 0$ on $\approxsol$ can be explained using Lemma \ref{lemma:linmeancurv}.  That is, one can imagine transformations of $\approxsol$ which rotate exactly one of its constituent hyperspheres or catenoidal necks by a rotation in $SO(n+2)$ while leaving all the other constituent hyperspheres and necks fixed. The associated approximate Jacobi field is of the form  $\chi q_V$ where $\chi$ is a cut-off function supported on one constituent of $\approxsol$ and $q_V$ is an exact Jacobi field for this constituent  as in the lemma.   It is known that the linear span of these functions approximates the small eigenspaces of $\linop$ well \cite[Appendix B]{kapouleas7}. 

An explicit representation of the Jacobi fields on the building blocks used to construct the approximate solution --- the hypersphere and the catenoid ---  will be now be given.  Then an explicit representation of the approximate Jacobi fields that will be used in the proof of Main Theorem 1 will be given at the beginning of the next section where the precise cut-off functions will be defined.  

\bigskip \noindent \itshape  1. Jacobi fields of the  hyperspheres. \upshape \smallskip
	
The linearized mean curvature operator of $S_\alpha$ is easily computed to be 
$$\mathcal L_a := \sin^{-2}(\alpha) \big(  \Delta_{\Sph^n} + n \big) \, .$$
Therefore, the Jacobi fields of $\mathcal L_a$ are simply the eigenfunctions of the $n$-sphere with eigenvalue $n$.  In the context of Lemma \ref{lemma:linmeancurv}, these can be derived by considering all non-trivial rotations of $S_\alpha$, namely the rotations generated by the vector fields
$$V_k := x^k \frac{\partial}{\partial x^0} - x^0 \frac{\partial}{\partial x^k} \: \: \mbox{for $k= 1, \ldots, n+1$} \, .$$  Taking the inner product of $V_k$ with $N_\alpha$ and restricting the resulting function to $S_\alpha$ the coordinate functions $x^k$ restricted to $S_\alpha$.

\bigskip \noindent \itshape 2. Jacobi fields of the catenoidal necks. \upshape \smallskip

The catenoidal necks of $\approxsol$ are catenoids $\Sigma$ in $\R^{n+1}$ that have been
re-scaled and embedded in $\Sph^n$ by the inverse of the canonical stereographic projection.  When the scale parameter is sufficiently small, it is enough to consider the Jacobi fields of $\Sigma$ treated as a hypersurface in $\R^{n+1}$, where the ambient metric is Euclidean rather than the with the metric induced by stereographic projection, and where the ambient isometries are the rigid motions of $\R^{n+1}$.  The linearized mean curvature operator of $\Sigma$ with respect to this background is easily computed to be
$$
\mathcal{L}_\Sigma  :=  \frac{1}{\phi^{n}} \, \frac{\partial}{\partial s} \left( \phi^{n-2}  \frac{\partial}{\partial s} \right) + \frac{1}{\phi^{2}}Ê\, \Delta_{\Sph^{n-1}} +  \frac{n(n-1)}{\phi^{2n}}
$$
in the standard parametrization of the catenoid. The isometries generating the relevant Jacobi fields of $\Sigma$ are as follows.  First, the ambient space $\R^{n+1} = \R \times \R^{n}$ possesses $n$ translations along the $\R^{n}$ factor and one translation in the $\R$ direction, which are generated by the vector fields
	$$V^{\text{trans}}_k := \frac{\partial}{\partial y^k} \:\: \mbox{ for $k=1,\ldots, n+1$}  \, .$$ 
	Then there are $n$ rotations of $\R \times \R^{n} $ that do not preserve the $\R$-direction, which are generated by the vector fields
	$$ V^{\text{rot}}_{1k} := y^{1} \frac{\partial}{\partial y^k} - y^k \frac{\partial}{\partial y^{1}} \:\: \mbox{ for $k = 2, \ldots, n+1$} \, .$$
	Finally, the motion of dilation in $\R^{n+1}$, though not an isometry, does preserve the mean curvature zero condition and is thus a geometric motion to which Lemma \ref{lemma:linmeancurv} can be applied.  Dilation is generated by the vector field
	$$V^{\text{dil}} := \sum_{k=1}^{n+1} y^k \frac{\partial}{\partial y^k} \, . $$  
	
	The Jacobi fields of $\mathcal{L}_\Sigma$ arising from the three classes of motions above can be found by the procedure of Lemma \ref{lemma:linmeancurv}.   One obtains the following non-trivial functions:
\begin{equation}
	\mylabel{eqn:catjacobi}
	\begin{aligned}
		J_1 (s) &:=  \langle N_\Sigma, V^{\text{trans}}_1 \rangle = \frac{\dot \phi(s)}{\phi(s)} \\
		J_k(s, \Theta) &:= \langle N_\Sigma, V^{\text{trans}}_k \rangle = - \frac{\Theta^k}{\phi^{n-1}(s)} \qquad k= 2, \ldots, n+1 \\
		J_{1k} (s, \Theta) &:= \langle  N_\Sigma, V^{\text{rot}}_{1k} \rangle =  \Theta^k \! \left( \frac{\psi(s)}{\phi^{n-1}(s)} + \dot \phi(s) \right) \qquad k = 2, \ldots, n+1 \\
		J_0(s) &:= \langle N_\Sigma, V^{\text{dil}} \rangle = \frac{\psi(s) \dot \phi(s)}{\phi(s)} - \frac{1}{\phi^{n-2}(s)} \, .
	\end{aligned}
\end{equation}
Note that the functions $J_k$ with $k \neq 0$ have odd symmetry with respect to the central sphere of $\Sigma$, i.e.~with respect to the transformation $s \mapsto -s$; while $J_{1k}$ and $J_0$ have even symmetry.  Also $J_1$ is bounded while $J_0$ has linear growth in dimension $n=2$ and is bounded in higher dimensions; $J_k$ decays like $\exp(-(n-1)|s|)$ for large $|s|$; and $J_{1k}$ grows like $\exp(|s|)$ for large $|s|$.

\subsection{The Linear Analysis}
\mylabel{sec:linest}

The most involved step in the application of the Banach space inverse function theorem  is the construction of the right inverse of the linear operator projected to a space orthogonal to the approximate co-kernel corresponding to the approximate Jacobi fields.  The purpose of this section of the paper is to explicitly define the projected linear operator and to find its right inverse on the appropriate Banach subspace of $\cobgs(\approxsol)$.   

The arguments that follow will require two carefully defined partitions of unity for the constituents of $\approxsol$.   First, for $s \in \{ 1, \ldots, L\}$ and $k \in \{ 0, \ldots, N_s - 1 \}$, define the smooth cut-off functions 
\begin{equation*}
	\eta^{sk}_{\mathit{neck}} (x) := 
	\begin{cases}
		1 &\qquad x \in \mathcal N^{sk} \\
		\mbox{Interpolation} &\qquad x \in \mathcal T^{sk} \\
		0 &\qquad \mbox{elsewhere}
	\end{cases}
\end{equation*}
and for $s \in \{ 1, \ldots, L\}$ and $k \in \{ 0, \ldots, N_s \}$, define the smooth cut-off functions
\begin{equation*}
	\eta^{sk}_{\mathit{ext}} (x) := 
	\begin{cases}
		1 &\qquad x \in \mathcal E_\eps^{sk} \\
		\mbox{Interpolation} &\qquad x \in \mbox{ any adjoining } \mathcal T^{s'k'}\\
		0 &\qquad \mbox{elsewhere}
	\end{cases}	
\end{equation*}
in such a way that $\sum_{s, k} \eta^{sk}_{\mathit{ext}} + \sum_{s, k} \eta^{sk}_{\mathit{neck}} = 1$.  In addition, one can assume that these cut-off functions are invariant under the group of symmetries $G_\Gamma$ and monotone in the interpolation regions.  Second, set $r_\tau := \max_{s, k} \{ r_{\eps_{sk}} \}$ and for $s \in \{ 1, \ldots, L\}$ and $k \in \{ 0, \ldots, N_s - 1 \}$ introduce the subsets $\mathcal N^{sk}(r) := \approxsol \cap B_{r} (p^\flat_{sk})$ where $r \in [  r_\tau, r_0]$.  This is a slightly enlarged version of the neck $\mathcal N^{sk}$ and its transition regions.  Define the smooth cut-off functions
\begin{equation*}
	\chi^{sk}_{\mathit{neck}, r} (x) := 
	\begin{cases}
		1 &\qquad x \in \mathcal N^{sk}(r) \\
		\mbox{Interpolation} &\qquad x \in  \mathcal N^{sk}(2r) \setminus \mathcal N^{sk}(r) \\
		0 &\qquad \mbox{elsewhere}
	\end{cases}
\end{equation*} 
and for $s \in \{ 1, \ldots, L\}$ and $k \in \{ 0, \ldots, N_s \}$, define the smooth cut-off functions
\begin{equation*}
	\chi^{sk}_{\mathit{ext}, r} (x) := 
	\begin{cases}
		1 &\qquad x \in \mathcal E^{sk} \setminus \displaystyle \left[ \, \bigcup_{\mbox{\scriptsize adjoining}}  \mathcal N^{s'k'}(2r) \right] \\[3ex]
		\mbox{Interpolation} &\qquad x \in \mbox{ any adjoining } \mathcal N^{s' k' }(2r) \setminus \mathcal N^{s' k' }(r) \\
		0 &\qquad \mbox{elsewhere}	\end{cases}	
\end{equation*}
so that once again $\sum_{s,k} \chi^{sk}_{\mathit{ext}, r} + \sum_{s,k} \chi^{sk}_{\mathit{neck}, r} = 1$ and invariance with respect to $G_\Gamma$ as well as the monotonicity in the interpolation regions hold.

\begin{rmk} The cut-off function $\chi^{sk}_{\mathit{ext}, r}$ should also be used for defining the graphing function of the perturbed hypersphere $\tilde S^{sk}_\alpha[\vec a_{sk}, \vec \sigma_{sk}]$ in equation \ref{eqn:distribeqn}.
\end{rmk}

The cut-off functions above and the considerations of Section \ref{sec:jacobi} leads to the definition of the space of approximate Jacobi fields of $\approxsol$ needed to construct the right inverse.  Fix $r \in [r_\tau, r_0]$ to be small but independent of $\tau$.  Let $x^t $ be the $t^{\mathit{th}}$ coordinate function for $t = 1, \ldots , n$.  For each $s, k$ recall that $R_{sk}[\vec \sigma_{sk}]$ is the $SO(n+2)$-rotation bringing $S^{sk}_\alpha [\vec \sigma_{sk}]$ into $S_\alpha$ 

\begin{defn}
\label{defn:apkerbasis}
Define the following objects.
\begin{itemize}
	\item The approximate Jacobi fields of $\approxsol$ are the functions
	\begin{equation*}
		\tilde q_{sk}^t  := \chi^{s,k}_{\mathit{ext}, r}  \cdot \left( x^t  \big\vert_{\tilde S^{sk}_{\alpha}[\vec a_{sk} , \vec \sigma_{sk}] } \circ \big( R_{sk}[\vec \sigma_{sk}] \big)^{-1} \right) \, .
	\end{equation*}
	Set $\tilde{\mathcal K} := \linspan_\R \big\{ \tilde q^t_{sk} \: : \: \mbox{all } s, t, k \big\}$. 
	
	\item The set of $G_\Gamma$-invariant approximate Jacobi fields of $\approxsol$ is
	\begin{align*}
		\tilde{\mathcal K}_{\mathit{sym}} &:= \linspan_\R \big\{ \tilde q \in \tilde{\mathcal K} \: : \: \tilde q \circ \rho = \tilde q \:\: \forall \: \rho \in G_\Gamma \big\}
	\end{align*}

	\item Denote the $L^2$-orthogonal complement of $\tilde{ \mathcal K}_{\mathit{sym}}$ in $C^{l,\beta}_{\delta, \mathit{sym}}(\approxsol)$ by $\big[ C^{l,\beta}_{\delta, \ast} (\approxsol) \big]^\perp$ and denote by
	$$\pi : C^{l,\beta}_{\delta, \mathit{sym} }(\approxsol) \rightarrow \big[ C^{l,\beta}_{\delta, \mathit{sym} }(\approxsol) \big]^\perp$$ 
	the corresponding $L^2$-projection operator.
\end{itemize}
\end{defn}

The preliminary notation is in place and the key result of this section of the paper can now be stated and proved.  The method that will be used to construct the right inverse involves patching together \emph{local solutions} of the equation $\linop^\perp (u) = f$ on each of the constituents of the approximate solution. 

\begin{prop}
	\mylabel{thm:seclinest}
	Suppose that the dimension of $\approxsol$ is $n\geq3$ and choose $\delta \in (2-n, 0)$. If $\tau$ and $ \| \vec \sigma \|$ are sufficiently small,  then the operator $$\linop^\perp  : \ctbgs (\approxsol) \rightarrow \big[ \cobgs (\approxsol) \big]^\perp$$ possesses a bounded right inverse $\rightinv$ satisfying the estimate
	$$| \rightinv(f) |_{\ctbg} \leq C | f |_{\cobg}$$
where $C$ is a constant independent of $\tau$ and $\vec \sigma$.  If the dimension of $\approxsol$ is $n=2$ then one can choose $\delta \in (-1, 0)$ and find a right inverse satisfying the estimate
	$$| \rightinv(f) |_{\ctbg} \leq C \eps^{\delta} | f |_{\cobg}$$
where  $\eps := \max_{s, k} \{ \eps_{sk} \}$ is the maximum of all the scale parameters of the necks of $\approxsol$ and $C$ is a constant independent of $\tau$ and $\vec \sigma$.
\end{prop}

\begin{proof}

The proof of this result  follows broadly the same plan as the proof of the analogous result in Butscher's paper \cite{me5}.  The significant differences occur in the first two steps, namely the derivation of the local solutions on the neck regions and the spherical regions of $\approxsol$.  The third step which consists of the estimates of the local solutions, is essentially unchanged.  Thus only the first two steps will be given here in full detail, and moreover only in the dimension $n\geq 3$ case.  (Due to the double indicial root of the principal part of $\linop$, which is the Laplacian, the proof in $n=2$ if slightly more complicated in a technical sense.  However, the modifications of the $n\geq3$ case needed to prove the $n=2$ case are the same as in \cite[Prop.~13]{me5} and will be left for the reader to carry out.)
	
Suppose that $f \in \big[ \cobgs(\approxsol) \big]^\perp$ is given.  The solution of the equation $\linop (u) = f$ will be constructed in three stages: local solutions on the neck regions will be found; then local solutions on the exterior regions will be found; and finally these solutions will be patched together to form an approximate solution which can be perturbed to a solution by iteration.   To begin this process, write $f = \sum_{s,k} f_{\mathit{ext}}^{sk} +  \sum_{s,k} f_{\mathit{neck}}^{sk}$ where $f_{\mathit{ext}}^{sk} := f \cdot \chi_{\mathit{ext}, r}^{sk}$ and $f_{\mathit{neck}}^{sk} := f \cdot \chi_{\mathit{neck}, r}^{sk}$.  Note that this set of functions inherits symmetries from $G_\Gamma$.  That is, for every $\rho \in G_\Gamma$ that fixes a spherical region or a neck region, the corresponding function is invariant under $\rho$; and for every $\rho \in G_\Gamma$ permuting two spherical regions or two neck regions (perhaps with a change of orientation) then the corresponding pair of functions are permuted (perhaps with an induced symmetry).  In the proof below, the case $G_\Gamma = \{ \mathit{Id} \}$ will actually be presented, since the more general case simply amounts to additional book-keeping.

\paragraph{Step 1.  Local solutions on the neck regions.}
Consider a given neck $\mathcal N := \mathcal N^{sk}$ and for the moment, drop the super- and sub-scripted $\mathit{sk}$ notation for convenience.  Let $K$ denote the canonical stereographic projection used to define the neck $\mathcal N$.  The subset $K \big( \mathcal N(r) \big) \subseteq \R \times \R^n$ is the union of two graphs over an annulus in the $\R^n$ factor, where the graphing functions are $y^1 = \pm \tilde F_\eps (\| \hat y \|)$ as defined in Section \ref{sec:glue} where $\eps := \eps_{sk}$ is the scale parameter of $\mathcal N$.  As such, it is a perturbation of the $\eps$-scaled catenoid $\eps \Sigma$.  Consequently, the function $f_{\mathit{neck}}:= f_{\mathit{neck}}^{sk}$ and the equation $\linop (u) = f_{\mathit{neck}}$ can be pulled back to  $\eps \Sigma$ which carries a perturbation of the catenoid metric $4 \eps^2 g_\Sigma$.  In this formulation, one can view  $f_{\mathit{neck}}$ as a function of compact support on $\eps \Sigma$.  The equation that will be solved in this step is $\frac{1}{4}\mathcal L_{\eps \Sigma} (u) = f_{\mathit{neck}}$ where $\frac{1}{4}\mathcal L_{\eps \Sigma}$ is the linearized mean curvature operator of $\eps \Sigma$ carrying exactly the metric $4\eps^2 g_\Sigma$.   

Let the catenoid be parametrized by $(s, \Theta) \mapsto (\eps \psi(s), \eps \phi(s) \Theta)$ as in equation \ref{eqn:gencat} and let $| \cdot |_{C^{k, \alpha}_\delta (\eps \Sigma)}$ denote the standard weighted $C^{l, \beta}_\delta$ norm on $\eps \Sigma$, so that 
$$
| u |_{{C}^{l, \beta}_\delta (\eps \Sigma)} : = \sum_{i=0}^l \big| (\eps  \cosh(s))^{-\delta+ i } \nabla^i u \big |_{0, \eps \Sigma} + \big[ (\eps  \cosh(s))^{-\delta+ l + \beta}\nabla^l u \big]_{\beta, \eps \Sigma}
$$ 
where the norms and derivatives correspond to the metric on $\eps \Sigma$.  A standard separation of variables argument shows that when $\delta \in (2-n, 0)$, the kernel of the operator $\frac{1}{4}\mathcal L_{\eps \Sigma} : C^{2, \beta}_\delta (\eps \Sigma) \rightarrow C^{0, \beta}_{\delta-2} (\eps \Sigma)$ consists of the linear span of the Jacobi fields $\{ J_{k} : k = 2, \ldots, n+1\}$. By the theory of the Laplace operator on asymptotically cylindrical manifolds, there is a solution $u_{\mathit{neck}} \in \ctbg(\eps \Sigma)$ that satisfies $\frac{1}{4} \mathcal L_{\eps \Sigma} (u_{\mathit{neck}}) =  \big( f_{\mathit{neck}} \big)^\sharp$, where $( \cdot )^\sharp$ denotes the $L^2$-orthogonal projection onto the $L^2$-orthogonal complement of the linear span of the Jacobi fields.  One can write 
$$ \big( f_{\mathit{neck}} \big)^\sharp = f_{\mathit{neck}} +  \sum_{t=1}^n (\lambda_{t}^{ +} \tilde Q^{ t, +} + \lambda_{t}^{-} \tilde Q^{t,-} )$$
where $\tilde Q^{t,+}$ and $\tilde Q^{t, -}$ are the pull-backs of the functions $\chi^{s, k+1}_{\mathit{neck}, 2r} \tilde q_{s, k+1}^t$ and  $\chi^{sk}_{\mathit{neck}, 2r} \tilde q_{sk}^t$ to $\eps \Sigma$, and
$$\lambda^\pm_t := - \frac{ \displaystyle \int_{\eps \Sigma} f_{\mathit{neck}} \cdot J_t}{ \displaystyle \int_{\eps \Sigma} \tilde Q^{t, \pm} \cdot J_t} \, .$$
One can check that $|\lambda^{\pm}_t| \leq C \eps^{\delta - 2 + n} |f|_{\cobg(\eps \Sigma)} \leq C \eps^{\delta - 2 + n} |f|_{\cobg}$ where $C$ is a constant independent of $\eps$.  Hence the estimate $|u_{\mathit{neck}} |_{\ctbg(\eps \Sigma)} \leq C |(f_{\mathit{neck}})^\sharp|_{\cobg(\eps \Sigma)} \leq C |f|_{\cobg}$ is valid, where $C$ is also independent of $\eps$.  Finally, the function $u_{\mathit{neck}}$ can be extended to all of $\approxsol$ by defining $\bar u^{sk}_{\mathit{neck}} := \chi^{sk}_{\mathit{neck}, r} \!\cdot u_{\mathit{neck}}$.  One has the estimate $|\bar u^{sk}_{\mathit{neck}} |_{\ctbg} \leq C  |f|_{\cobg}$.

\paragraph{Step 2. Local solutions on the exterior regions.}  Consider a given spherical region $\mathcal E := \mathcal E^{sk}$ and again, drop the super- and sub-scripted $\mathit{sk}$ notation for convenience. Given the local solution $\bar u_{\mathit{neck}}$ constructed in the previous step, choose a small $\kappa \in (0,1)$ and define $\hat f_{\mathit{ext}} := \hat  f_{\mathit{ext}}^{sk}$ where
$$
\hat  f_{\mathit{ext}}^{sk}  : =  \chi_{\mathit{ext}, \kappa r}^{sk}  \Biggl( f -  \linop \Biggl(\, \sum_{s',k'} \bar u^{s'k'}_{neck} \Biggr)  \Biggr) \, .
$$
This function vanishes within an $\eps := \eps_{sk}$-independent distance from the union of all the neck regions associated to $\mathcal E$.  Therefore one can determine without difficulty $
|  \hat  f_{\mathit{ext}}   |_{C^{0, \beta}} \leq C_\kappa   |  f |_{\cobg}
$ for some constant $C_\kappa$ that depends on $\kappa$ and $\delta$.  Here, $| \cdot |_{C^{0,\beta}}$ is the un-weighted Schauder norm.

The function $\hat f_{\mathit{ext}}$ can be viewed as a function of compact support on the perturbed hypersphere $\tilde S_\alpha[\vec a, \vec \sigma]$.  Since $\tilde S_\alpha[\vec a, \vec \sigma]$ is a normal graph over the hypersphere $S_\alpha[ \vec \sigma]$, this function can be pulled back to the hypersphere $S_\alpha[\vec \sigma]$ and then to the standard hypersphere $S_\alpha$ vanishing in the neighbourhood of certain points $\{ p_1, \ldots, p_K \}\subseteq  S_\alpha$ where $S_\alpha$ is attached to other perturbed hyperspheres through necks.  The metric  carried by $S_\alpha$ in this identification is a perturbation of the standard induced metric $\sin^2(\alpha) g_{\Sph^{n}}$.  However, the equation that will be solved here is $\mathcal L_{\alpha} (u) = \hat f_{\mathit{ext}}$ up to projection onto the approximate co-kernel, where $\mathcal L_{\alpha}$ is the linearized mean curvature operator of $S_\alpha$ when it carries the un-perturbed metric $\sin^2(\alpha) g_{\Sph^{n}}$. 

Compute the quantities $\mu_{tt'} :=  \int_{S_\alpha}\tilde q_t \cdot  x^{t'}\big|_{S_\alpha}$ where $\tilde q_t \in \tilde{\mathcal K}$ are the Jacobi fields supported on $\tilde S_\alpha[\vec a, \vec \sigma]$ and pulled back to $S_\alpha$ and $x^t \big|_{S_\alpha}$ are the coordinate functions restricted to $S_\alpha$. Set $\mu^{tt'}$ equal to the components of the inverse of the matrix whose components are $\mu_{tt'}$.   (It can easily be verified that this matrix is invertible because $\tilde q_t$ is almost equal to $x^t\big|_{S_\alpha}$ and these functions form an $L^2$-orthogonal set.)  Now
$$\big( \hat f_{\mathit{ext}} \big)^\sharp := \hat f_{\mathit{ext}}  - \sum_{t, t'} \tilde q_{t'} \cdot \mu^{tt'} \cdot \int_{S_\alpha} \hat f_{\mathit{ext}} \cdot x^{t}\big|_{S_\alpha} $$ 
is orthogonal to the coordinate functions restricted to $S_\alpha$.   The equation $\mathcal L_\alpha (u_{\mathit{ext}}) = \big( \hat f_{\mathit{ext}} \big)^\sharp$ can now be solved for $u_{\mathit{ext}}$ in $C^{2,\beta}(S_\alpha)$.  The solution satisfies the estimate $|u_{\mathit{ext}}|_{C^{2,\beta}(S_\alpha)} \leq C_\kappa |(\hat f_{\mathit{ext}})^\sharp|_{\cobg(S_\alpha)}$.  

The function $u_{\mathit{ext}}$ can now be extended to all of $\approxsol$ as follows.  Suppose that $u_{\mathit{ext}}(p_j) := a_j$ for $j = 1, \ldots K$ and let $A_{\mathit{ext}} : S_\alpha \rightarrow \R$ be a smooth function that is locally constant near each $p_j$ satisfying $A_{\mathit{ext}}(p_j) = a_j$.  Then $u_{\mathit{ext}} = A_{\mathit{ext}} + \tilde u_{\mathit{ext}}$ where $\tilde u_{\mathit{ext}}$ is smooth function satisfying $\tilde u_{\mathit{ext}} = \mathcal O( \mathrm{dist}( \cdot , p_j) )$ near each $p_j$.   For $j=1, \ldots, K$, let $\mathcal J_j$ be the linear combination of the Jacobi fields $J_0$ and $J_1$ defined on the neck adjoining $S_\alpha$ at the point $p_j$ that has limit $a_j$ on the end of this neck attached to $S_\alpha$ and has limit zero on the other end of this neck.  Note that $\mathcal J_j = a_j + \tilde{\mathcal J_j}$ where $\tilde{\mathcal J}_j = \mathcal O( \mathrm{dist}( \cdot , p_j) )$ in the part of this neck overlapping with $S_\alpha$.    Now define
$$\bar u^{sk}_{\mathit{ext}} := \eta_{\mathit{ext}} u_{\mathit{ext}} + \sum_{j=1}^K  \eta_{\mathit{neck}}^{j} \mathcal J_j  \, .$$
The extended function $\bar u_{\mathit{ext}}$ satisfies the estimate $|\bar u_{\mathit{ext}}|_{\ctbg} \leq C_\kappa | ( \hat f_{\mathit{ext}})^\sharp |_{\cobg}$ for some constant $C_\kappa$ depending  on $\kappa$ and $\delta$ but not $\eps$.

\paragraph*{Step 3. Estimates and convergence.}  Local solutions $\bar u^{sk}_{\mathit{neck}}$ and $\bar u^{sk}_{\mathit{ext}}$ on the neck regions and on the exterior regions, respectively, have been found and extended to all of $\approxsol$.  Define the function $\bar u := \sum_{s, k} \bar u^{sk}_{\mathit{neck}} + \sum_{s, k} \bar u^{sk}_{\mathit{ext}}$.   Then a long computation yields
\begin{equation}
	\label{eqn:cvgce}
	\begin{aligned}
		\linop^\perp (\bar u) - f &= \pi \circ \sum_{s,k} \biggl( [ \linop, \eta^{sk}_{\mathit{ext}} ] ( \tilde u^{sk}_{\mathit{ext}}) +  \eta_{\mathit{ext}}^{sk} \big( \linop - \mathcal L_\alpha \big)( u^{sk}_{\mathit{ext}})  + \sum_j [ \linop, \eta^{sk}_{\mathit{neck}}]( \tilde{ \mathcal J}_j ) \\ 
		&\qquad \qquad  +  \sum_j \eta^{sk}_{\mathit{neck}} \big( \linop - \tfrac{1}{4} \mathcal L_{\eps_{sk} \Sigma} \big) ( \tilde{\mathcal J}_j)   + \chi^{sk}_{\mathit{neck}, \kappa r} \big(  \linop - \tfrac{1}{4} \mathcal L_{\eps_{sk} \Sigma} \big) (\bar u^{sk}_{\mathit{neck}}) \biggr) \phantom{\sum_{s,k}} \hspace{-4ex}
\end{aligned}
\end{equation}
where  $[\mathcal L, \eta] (u) :=\mathcal L (\eta u) - \eta \mathcal L(u)$.  Each term in \eqref{eqn:cvgce} can now be shown to have small $\cobg$ norm using the same estimation technique as in \cite{me5}.  That is: the $[\linop, \eta^{sk}_\ast]$ terms are small because $ \tilde u^{sk}_{\mathit{ext}}$ and $\tilde{\mathcal J}_j$ have stronger decay (than $u^{sk}_{\mathit{ext}}$ and $\mathcal J_j$) in the support of gradient of $\eta^{sk}_\ast$; while the $\linop - \mathcal L_\ast$ terms are small because $\linop$ differs very little from both $\mathcal L_\alpha$ and $\frac{1}{4} \mathcal L_{\eps_{sk} \Sigma}$ in the regions upon which the arguments of these operators are supported. The result is that if all $\kappa$ parameters are sufficiently small to begin with, and then all $\eps$ parameters are made as small as needed, then it is possible to achieve $|\linop^\perp (\bar u) - f |_{\cobg} \leq \frac{1}{2} |f|_{\cobg}$.  The estimate $|\bar u|_{\ctbg} \leq C |f|_{\cobg}$ can also be found using the same techniques.  The proof of the proposition now follows by a standard iteration argument.
\end{proof}

\subsection{The Non-Linear Estimates}
\mylabel{subsec:nonlinest}

Invoking the Banach space inverse function theorem to solve the equation $\pi \circ \difop(f) = 0$ requires two more estimates in addition to the right inverse and linear estimate from the previous section.  It is necessary to show that $\pi \circ \difop(0)$ has small $\cobg$ norm; and it is necessary to show that $\Dif( \pi \circ \difop) (f)  -   \linop^\perp$ can be made to have small $\ctbg$-operator norm if $f$ is chosen to have sufficiently small $\ctbg$ norm.  These two estimates are in most respects identical to those computed in Butscher's paper \cite{me5} and will thus only be sketched here.

\begin{prop}
	\mylabel{prop:error}
	The quantity $\pi \circ \difop(0)$ satisfies the following estimate.  If $\tau$ and $ \vec \sigma$ are sufficiently small, then there exists a constant $C$ independent of $\tau$ and $\vec \sigma$ so that 
	\begin{equation}
		\mylabel{eqn:error}
		\vert \pi \circ \difop(0) \vert_{\cobg} \leq C r_\eps^{2 - \delta}
	\end{equation}
	where $\eps :=  \max \{ \eps_{sk} \}$ is the maximum of all the scale parameters of the necks of $\approxsol$ and $r_\eps := \eps^{(3n-3)/(3n-2)}$.
\end{prop}

\begin{proof} 

The estimate \eqref{eqn:error} can be computed as in \cite{me5} by verifying separately in the spherical regions, in the transition regions, and in the neck regions of $\approxsol$ that the mean curvature is sufficiently close to $H_\alpha$, except with one significant modification in the first of these computations.  To see this, consider one fixed spherical region $\mathcal E^{sk}$ pulled back to the standard hypersphere $S_\alpha$.  The expression for the mean curvature of a normal graph over $S_\alpha$ when the graphing function is $G := G_{sk}$, as given in \cite{me5}, reads
\begin{equation}
	\mylabel{eqn:normgrmeancurv}
	\begin{aligned}
		H \big( \exp( G N_\alpha)(S_\alpha) \big) - H_\alpha &=  \frac{-\Delta G +  n \sin(\alpha + G) \cos(\alpha + G)}{A \sin(\alpha +  G)}\\[1ex]
		&\qquad  - \frac {\nabla^2 G (\nabla G, \nabla G) - \cos(\alpha + G) \sin(\alpha+G) \| \nabla G \|^2}{A^3 \sin(\alpha +  G)}  - H_\alpha
	\end{aligned}
\end{equation}
where $\nabla$ and $\Delta$ are the covariant derivative and the Laplacian of the standard metric of $\Sph^n$, and $A = \big( \sin^2(\alpha + G) + \| \nabla G \|^2  \big)^{1/2}$.   By formally expanding this expression in when $G$ is small as in \cite{me5}, one finds that the largest term is $- ( \Delta + n)(G)$.  The quantity $(\Delta + n) (G)$ equals a term in $\tilde{\mathcal K}_{\mathit{sym}}$ by definition.  Moreover, as in \cite{me5} the expansion of $G$ given in equation \eqref{eqn:distribeqnasym} and the estimate $\| \vec a_{sk} \| = \mathcal O(\eps^{n-1})$ yields $| H(\approxsol) - H_\alpha + (\Delta + n) (G)|_{C^0_{2-\delta} (\mathcal E_\eps )} \leq C r_\eps^{2-\delta}$.  Under $L^2$ projection, $(\Delta + n) (G)$ disappears so that the desired estimate follows.
\end{proof}  

\begin{prop}
    \mylabel{prop:nonlin}
    The linearized mean curvature operator satisfies the following general estimate.  If $\tau$ and $ \vec \sigma$ are sufficiently small and $f \in \ctbgs(\approxsol)$ has sufficiently small $\ctbg$ norm, then there exists a constant $C$ independent of $\tau$ and $\vec \sigma$ so that
    \begin{equation}
        \mylabel{eqn:nonlin}
        \big\vert \Dif (\pi \circ \Phi_{\alpha, \tau \vec \sigma} ) (f) (u) - \linop^\perp ( u ) \big\vert_{\cobg} \leq C \eps^{\delta - 1} |f|_{\ctbg}Ê\, \vert u \vert_{\ctbg }
    \end{equation}
     for any function $u \in \ctbgs(\approxsol)$, where $\eps :=  \max \{ \eps_{sk} \}$ is the maximum of all the scale parameters of the necks of $\approxsol$.
\end{prop}

\begin{proof}
This follows from a scaling argument exactly as in \cite{me5}.
\end{proof}

\subsection{The Solution of the Non-Linear Problem up to Finite-Dimensional Error}

The linear and non-linear estimates derived in the previous sections can now be combined to solve the equation $\pi \circ \difop(f) = 0$ using the Banach space inverse function theorem up to a finite-dimensional error term contained in the kernel of $\pi$.  

\begin{prop}
	\mylabel{prop:secperturb}
	If $\tau$ and $\vec \sigma $ are sufficiently small, then there exists $f_{\tau, \vec \sigma} \in \ctbg (X)$ satisfying $\pi \circ \difop (f_{\alpha, \tau,\vec \sigma}) = 0$ and there exists a constant $C$ independent of $\tau$ and $\vec \sigma$ so that 
	$$| f_{\tau, \vec \sigma} |_{\ctbg} \leq C(\eps) \cdot r_\eps^{2-\delta}$$
	where $C(\eps) = \mathcal O(1)$ in dimension $n\geq 3$ and $C(\eps) = \mathcal O(\eps^{\delta})$ in dimension $n = 2$.   Here $\eps :=  \max \{ \eps_{sk} \}$ is the maximum of all the scale parameters of the necks of $\approxsol$ and $r_\eps := \eps^{(3n-3)/(3n-2)}$.  As a result, the hypersurface obtained by deforming $\approxsol$ in the normal direction by an amount determined by $f_{\tau, \vec \sigma}$ is embedded if $\approxsol$ is embedded.
\end{prop}

\begin{proof}
The linearization of $\pi \circ \difop$ at zero is $\Dif ( \pi \circ \difop )(0) = \linop^\perp$ and this operator possesses a bounded right inverse by Proposition \ref{thm:seclinest}.  The Banach space inverse function theorem can thus be applied to the equation $\pi \circ \difop (f) = 0$ provided that the three fundamental estimates \eqref{eqn:iftestone}, \eqref{eqn:iftesttwo} and \eqref{eqn:iftestthree} described in Section \ref{sec:jacobi} can be established.  The construction of the right inverse and its bound in Proposition \ref{thm:seclinest} constitutes the first of these estimates. One has
$$| \rightinv (f) |_{\ctbg} \leq C_L(\eps) | f |_{\cobg} \, ,$$
where $C_L(\eps) = \mathcal O(\eps^{\delta})$ in dimension $n=2$ and $C_L(\eps) = \mathcal O(1)$ in higher dimensions.  Now in order to achieve the bound
$$| \Dif (\pi \circ \difop) (f)(u) -  \linop^\perp (u) |_{\cobg} \leq \frac{1}{2 C_L(\eps)} |u|_{\ctbg} \, ,$$
for any $u \in \ctbgs(\approxsol)$, one must have $|f|_{\ctbg} \leq R$ where $R = \mathcal O(\eps^{1-2\delta})$ in dimension $n=2$ or $R = \mathcal O(\eps^{1-\delta})$ in higher dimensions.  The inverse function theorem now asserts that a solution $f := f_{\tau, \vec \sigma}$  of the deformation problem can be found if $\approxsol$ satisfies the estimate 
$$| \pi \circ \difop(0)  |_{\cobg} \leq \frac{R}{2 C_L } = \mathcal O(\eps^{1 - 3 \delta}) \, .$$
But since Proposition \ref{prop:error} asserts that 
$$| \pi \circ \difop(0)  |_{\cobg} = \mathcal O(\eps^{(2-\delta)(3n-3)/(3n-2)})\, ,$$ 
this is true so long as $\eps$, $\tau$ and $\| \vec \sigma \|$ are sufficiently small and $\delta$ is chosen properly.

As a further consequence of these estimates, the Banach space inverse function theorem asserts that the solution of the equation $\pi \circ \difop(f_{a, \tau, \vec \sigma}) = 0$ satisfies the estimate 
$$| f_{\tau, \vec \sigma} |_{\ctbg} = \mathcal O \big(  C_L(\eps) \cdot  \eps^{(2-\delta)(3n-3)/(3n-2)}  \big)$$
which is much smaller that $\eps$.  Therefore the size of the perturbation of $\approxsol$ created by the normal deformation of magnitude $f_{\tau, \vec \sigma}$ is much smaller than the width of $\approxsol$ at its narrowest points, i.e.~in the neck regions where the width is $\mathcal O(\eps)$.  Thus $\approxsol$ remains embedded under this normal deformation.  
\end{proof}

\section{Solution of the Finite-Dimensional Problem}
\label{sec:balancing}

\subsection{The Balancing Map}

Proposition \ref{prop:secperturb} shows that the equation $\difop (f) = 0$ can be solved up to a finite dimensional error term; i.e.~a function $f_{\tau, \vec \sigma} \in  \ctbgs(\approxsol)$ can be found so that only the $L^2$-projection of $\difop( f_{\tau, \vec \sigma})$ to the subspace $\tilde{ \mathcal K}_{\mathit{sym}}$ fails to vanish identically.   Since there is such a function for each sufficiently small $\vec \sigma \in \mathcal D_{\Gamma}$ and thus one can consider the map $\vec \sigma \mapsto \difop( f_{\tau, \vec \sigma})$ as a function of $\vec \sigma$.  It will now be shown that under the hypotheses of Main Theorem \ref{result1} there is a special choice of $\vec \sigma$ for which $\difop(f_{\tau, \vec \sigma})$ vanishes completely.  Therefore the solution $f_{\tau, \vec \sigma}$ for this choice of $\vec \sigma$ yields the desired deformation of $\approxsol$ into an exactly CMC hypersurface.  In order to show how this special value of $\vec \sigma$ is found, one must first understand in greater detail the relationship between $\vec \sigma$ and the quantity $(\mathrm{id} - \pi) \circ  \difop (f_{\tau, \vec \sigma})$ where $\pi$ is the $L^2$-projection onto $\tilde{\mathcal K}_{\mathit{sym}}^{\perp}$.

To analyze this relationship properly,the first step to re-phrase the problem slightly.  Let $\tilde q_1, \ldots, \tilde q_N$ be a basis for $\tilde{\mathcal{K}}_{\mathit{sym}}$ constructed from an $L^2$-orthonormal basis for the eigenfunctions of $\Delta_S + n$ on $S_\alpha$ as in Definition \ref{defn:apkerbasis}.  Next, define a slightly different set of functions $\tilde q_1', \ldots, \tilde q_N'$ obtained from the $\tilde q_1, \ldots, \tilde q_N$ by replacing each $\chi_{\mathit{ext}, r}$ appearing in the definition of a $\tilde q_j$ with $\chi_{\mathit{ext}, r_\eps}$.  As usual, here $\eps := \max \{ \eps_{sk} \}$ and $r_\eps := \eps^{(3n-3)/(3n-2)}$  where $\eps_{sk}$ is the scale parameter of the $k^{\mathit{th}}$ neck along the geodesic $\gamma_s$.  Now one can decompose
$$(\mathrm{id} - \pi) \circ  \difop (f_{\tau, \vec \sigma}) = \sum_{i, j =1}^N M^{ij}(\vec \sigma) \cdot B_i(\vec \sigma) \cdot \tilde q_j$$
where $B_i : \mathcal D_{\Gamma} \rightarrow \R$ are real-valued functions of the displacement parameters defined by 
\begin{equation}
	\label{eqn:balmap}
	B_j(\vec \sigma) := \int_{\phi_{f_{\tau, \vec \sigma}} (\approxsol) }  \difop (f_{\tau, \vec \sigma}) \cdot \tilde q_j'
\end{equation}
and $M^{ij}(\vec \sigma)$ are the coefficients of the inverse of the matrix with coefficients $\int_{\phi_{f_{\tau, \vec \sigma}} (\approxsol) } \tilde q_i \cdot \tilde q_j'$.   One can check that this matrix is a small perturbation of the identity matrix and is indeed invertible. Also, $\phi_{f_{\tau, \vec \sigma}}$ in \eqref{eqn:balmap} is the normal deformation corresponding to $f_{\tau, \vec \sigma}$.

\begin{defn}
	\mylabel{defn:balancingmap}
	The \emph{balancing map} of $\approxsol$ with respect to the chosen basis $\{ \tilde q_1,\ldots \tilde q_N \}$ of $\tilde {\mathcal{K}}_{\mathit{sym}}$ is the function $B_{\tau} : \mathcal D_{\Gamma} \rightarrow \R^{K}$ given by 
	\begin{equation*}
		B_{\tau} (\vec \sigma) : = \big( B_1 (\vec \sigma), \ldots B_N(\vec \sigma) \big) \, ,
	\end{equation*}
	and each $B_j : \mathcal{D}_\Gamma \rightarrow \R$ is defined as in \eqref{eqn:balmap}.\end{defn}	

In terms of the balancing map, what remains to be done in order to prove Main Theorem \ref{result1} is to find a value of $\vec \sigma$ for which $B_{\tau}(\vec \sigma) = 0$.

\subsection{Approximating the Balancing Map} 
\mylabel{subsec:balmap}

The balancing map can be better understood by deriving an approximation of the map which is  independent of $f_{\tau, \vec \sigma}$.  To see how this is done, note that each $\tilde q_j'$ is a $G_\Gamma$-invariant linear combination of the approximate Jacobi fields in Definition \ref{defn:apkerbasis}, each of which is supported on exactly one of the constituent perturbed hyperspheres of $\approxsol$.   Thus it suffices to find a good approximation of the function 
$$B : \vec \sigma \mapsto \int_{\phi_{f_{\tau, \vec \sigma}} (\approxsol)} \difop(f_{\tau, \vec \sigma}) \cdot \tilde q' $$ 
where $\tilde q ' = \sum_{t=1}^n a_t \chi_{\mathit{ext}, r_\eps}^{sk} q_{sk}^t$ and $q_{sk}^t$ are the Jacobi fields of this hypersphere as in Definition \ref{defn:apkerbasis}.

Suppose that  the $(s,k)$-perturbed hypersphere in $\approxsol$  is a perturbation of $S_\alpha^{sk}[\vec \sigma_{sk}] := \big( R_{sk}[\vec \sigma_{sk}] \big)^{-1} (S_\alpha \setminus \{ p_1, \ldots, p_K \} )$.  Recall that the infinitesimal generator of rotation associated to $q_{sk}^t$ is the vector field
\begin{equation}
	\label{eqn:genvecflds}
	Y_{sk}^t := \big( R_{sk}[\vec \sigma_{sk}] \big)^{-1}_\ast  \bigl[ Y^t \circ \big( R_{sk}[\vec \sigma_{sk}] \big) \bigr] \qquad \mbox{where} \qquad Y^t := x^t \frac{\partial}{\partial x^0} - x^0 \frac{\partial}{\partial x^t} \, .
\end{equation}
Set $ Y := \sum_{t=1}^n a_t Y^t_{sk}$ and $q := \sum_{t=1}^n a_t q^t$.  An analysis of the function $B$ reveals the following.

\begin{prop}
	\mylabel{prop:balmapapprox}
	Let $\tilde q$ be as above.  Then the function $B$ can be decomposed as
	$$B (\vec \sigma) = \mathring B (\vec \sigma)  + E (\vec \sigma) \, .$$ 
	In this decomposition, $\mathring B : \mathcal D_{\Gamma} \rightarrow \R$ is defined as follows.   Suppose that  $p_j := \exp_{p_0}( \alpha T_j )$ where $T_j$ is the unit vector in $T_{p_0} \Sph^{n+1}$ tangent to the geodesic connecting $p_0$ and $p_j$.  Then
	\begin{equation}
		\label{eqn:balmapapprox}
		\mathring B (\vec \sigma) :=  \sum_{j=1}^K \omega \eps_j^{n-1} \langle T_j, Y \rangle \, .
	\end{equation}
	Furthermore, $E : \mathcal D_\Gamma \rightarrow \R$ satisfies the estimate 
	$$\Vert E (\vec \sigma) \Vert_{C^2} \leq C r_\eps^{n}$$
	where $C$ is a constant independent of $\tau$ and $\vec \sigma$.  \end{prop}

\begin{proof}
The integral defining $B$ is invariant under rotation, so that one can assume that $R_{sk}[\vec \sigma_{sk}]$ is the identity so that $B$ corresponds to the standard punctured hypersphere $S_\alpha \setminus \{p_1, \ldots, p_K\}$, which shall be denoted here by $\tilde S^0[\vec \sigma_0]$.  Denote the nearest neighbours of $\tilde S^0[\vec \sigma_0]$ by $\tilde S^j[\vec \sigma_j]$.  Let these be connected to $\tilde S^0[\vec \sigma_0]$ through necks $\mathcal N_j$ with scale parameters $\eps_j$.   Finally, denote by $D_j$ the disk $\{ (0, \hat y) \in \R \times \R^n : \| \hat y \| \leq \eps_j \}$ pushed forward by the canonical coordinate chart corresponding to the neck $\mathcal N_j$ and let $c_j = \partial D_j$.  In other words, $c_j$ is the smallest sphere in the throat of $\mathcal N_j$ and $D_j$ is an $n$-dimensional cap for $c_j$. Denote by $\mathcal N_j^-$ the component of $\mathcal N_j \setminus c_j$ that is attached to $\tilde S^0[\vec \sigma_0]$ at the point $p_j$ and set $\mathcal N^- := \mathcal N_1^- \cup \cdots \cup  \mathcal N_K^- $.

Consider now the integral defining $B$.  The idea is to apply the Korevaar-Kusner-Solomon and Kapouleas balancing formula \eqref{eqn:kksbalancing} for the integral of $\difop(f_{\tau, \vec \sigma}) := H_{\phi_{f_{\tau, \vec \sigma }} (\approxsol)} - H_\alpha$ to replace this integral with a sum of boundary terms.  Then the fact that $f_{\tau, \vec \sigma}$ is small gives an approximate expression that pertains solely to the initial configuration of hyperspheres.   These calculations are
\begin{align}
	\label{eqn:balcalc}
	 B (\vec \sigma ) &= \int_{\phi_{f_{\tau, \vec \sigma}}( \tilde S^0[\vec \sigma_0] \cup  \mathcal N^-) }  \difop (f_{\tau, \vec \sigma}) \cdot \chi_{\mathit{ext}, r_\eps} \cdot q \notag \\
	&= \int_{\phi_{f_{\tau, \vec \sigma}}( \tilde S^0[\vec \sigma_0] \cup  \mathcal N^-) }  \difop (f_{\tau, \vec \sigma})  \cdot q + \mathcal O (r_\eps^{n} ) \notag  \\
	&= \sum_{j=1}^K \int_{\phi_{f_{\tau, \vec \sigma}}(c_j)} \big\langle \nu_j , Y\big\rangle  + \mathcal O (r_\eps^{n} ) \notag \\
	&= \sum_{j=1}^K \int_{c_j} \big\langle \nu_j ,Y  \big\rangle  + \mathcal O (\eps^{\delta + n-1} |f |_{\ctbg} ) + \mathcal O(r_\eps^n)
\end{align}
where $\nu_j$ is the outward unit normal vector field of $c_j$ tangent to $\mathcal N^{j,-}$.   Note that the $\int_{D_j}$ terms in the Korevaar-Kusner-Solomon and Kapouleas balancing formula have been absorbed into the error term.  This is because when $\eps$ is small then these quantities are much smaller than the $\int_{c_j}$ terms.  

Finally, the calculation of the integrals $\int_{c_j} \langle \nu_j , Y \rangle$ in \eqref{eqn:balcalc}  can be carried out in the stereographic coordinate chart used to define $\mathcal N_j$.  This is very straightforward and yields a quantity proportional to the $(n-1)$ dimensional area of $c_j$ in the form $\omega \eps_j^{\, n-1} \langle  \dot \gamma_j(\alpha) , Y \rangle$ where $\gamma_j$ is the geodesic from $p_0$ to $p_j$ while $\omega$ is a constant independent of $\eps$.  But since $Y$ is a Killing field, this quantity remains constant along $\gamma_j$ and can thus be transported to $p_0$.  The desired formul\ae\ follow.
\end{proof}

The calculations of the previous proposition show that the balancing map consists of a collection of principal terms like \eqref{eqn:balmapapprox}, one for each perturbed hypersphere in $\approxsol$, plus error terms which are of size $\mathcal O(r_\eps^n)$.  The principal term corresponding to a given perturbed hypersphere depends on the displacement parameter of this perturbed hypersphere, as well as on the displacement parameters of all neighbouring perturbed hyperspheres.   It is important to realize that the principal term depends on no other displacement parameters.  As defined in the introduction, an initial configuration of hyperspheres is \emph{balanced} if $\mathring B (0) = 0$.

A formula for the derivative of the approximate balancing map at $\vec \sigma = 0$ will also be needed in the sequel.  To this end, a more explicit formula illustrating the dependence of $\mathring B$ on $\vec \sigma$ is needed.  In what follows, denote once again the $(s,k)$-perturbed hypersphere by $\tilde S^0[\vec \sigma_0]$ and suppose it is centered on $p_0[\vec \sigma_0]$.  As before, one can assume that $\tilde S^0[0]$ is a perturbation of the punctured hypersphere $S_\alpha \setminus \{p_1, \ldots, p_K\}$.  Denote the nearest neighbours of $\tilde S^0[\vec \sigma_0]$ by $\tilde S^j[\vec \sigma_j]$ for $j = 1, \ldots, K$ and suppose these are centered at $p_j[\vec \sigma_j]$ with $p_j[0] = p_j$.  Denote the geodesic connecting $p_0[\vec \sigma_0]$ to $p_j[\vec \sigma_j]$ by $\gamma_j[\vec \sigma_0, \vec \sigma_j]$.   Let the tangent vectors of $\gamma_j[\vec\sigma_0, \vec\sigma_j]$ at $p_0[\vec\sigma_0]$ and $p_j[\vec \sigma_j]$ be $T_j[\vec \sigma_0, \vec \sigma_j] := \csc(\tau_j + 2 \alpha) \big( p_j[\vec \sigma_j] - p_0[\vec \sigma_0] \cos(\tau_j + 2 \alpha) \big)$ and $T_j^{\, \prime}[\vec \sigma_0, \vec \sigma_j] := \csc(\tau_j + 2 \alpha) \big( p_0[\vec \sigma_0] - p_j[\vec \sigma_j] \cos(\tau_j + 2 \alpha) \big)$.

The map $\mathring B$ can be related to $\vec \sigma$ explicitly as follows.  First, the relationship between the scale of the neck used to connect two perturbed hyperspheres and their separation, established in equation \eqref{eqn:match}, gives $\eps_j := \eps(  \tau_j )$ where $\tau_j :=  \mathrm{dist}(p_0[\vec \sigma_0], p_j[\vec \sigma_j] ) - 2 \alpha $ and $\eps: \R \rightarrow \R$ is some universal function determined via the matching process.   Recall further that  $\tilde S^j[\vec \sigma_j] = \mathcal W_{\vec \sigma_j} \big( S^j[0] \big)$ for $j=0, \ldots, K$ where $\mathcal W_{\vec \sigma_j}$ is the unique $SO(n+2)$-rotation that coincides with the exponential map at $p_j[0]$ in the direction of $\vec \sigma_j$.  Moreover, the basis of infinitesimal generators of the rotations of $\tilde S^0[\vec \sigma_0]$ are of the form $\big( \mathcal W_{\vec \sigma_0} \big)_\ast Y \circ \mathcal W_{\vec \sigma_0}^{-1}$ where $Y$ is a linear combinations of the vector fields given in \eqref{eqn:genvecflds}.   One therefore obtains the formula
\begin{equation}
	\label{eqn:balmapdependence}
	\mathring B (\vec \sigma_0, \vec \sigma_1, \ldots, \vec \sigma_K)  = \sum_{j=1}^K \omega \eps_j^{n-1} \big \langle \big( \mathcal W_{\vec \sigma_0} \big)_\ast^{-1} T_j [\vec \sigma_0, \vec \sigma_j], Y \big \rangle \Big|_{p_0[0]} \, .
\end{equation}
This illustrates completely how $\mathring B$ depends only on $\vec\sigma_0$ and $\vec \sigma_j$ for $j= 1, \ldots, K$ and on no other displacement parameters.

\begin{prop}
	\label{prop:balmapderiv}
	Let $V$ be a tangent vector at the origin in the space of displacement parameters.  Suppose that $V_0 \in T_{p_0[0]} \Sph^{n+1}$ is the component of  $\, V$ corresponding to the perturbed hypersphere $\tilde S^0[\vec \sigma_0]$ and $V_j \in  T_{p_j[0]} \Sph^{n+1}$ are the components of $V \!$ corresponding to the nearest neighbours $\tilde S^j[\vec \sigma_j]$ for $j = 1, \ldots, K$.  Then
	\begin{equation}
		\label{eqn:balmapderiv}
		\begin{aligned}
			\Dif  \mathring B(0) (V) &= - \sum_{j=1}^K (n-1) \, \omega\,  \eps_j^{n-2} \dot \eps(\tau_j)  \Big( \big\langle V_0^{\|_j} , Y \big\rangle -   \tan (\tau_j + 2 \alpha) \big\langle \big[ V_j^\sharp \big]^{\|_j} , Y \big\rangle \Big) \\
			&\qquad - \sum_{j=0}^K \omega \, \eps_j^{n-1}  \Big(  \big\langle V_0^{\perp_j}, Y\big\rangle  - \tan( \tau_j + 2 \alpha) \big\langle \big[ V_j^\sharp \big]^{\perp_j} ,Y  \, \big\rangle  \Big)
		\end{aligned}
	\end{equation}
	where $X^{\|_j}$ and $X^{\perp_j}$ denote the projections of a vector $X$ parallel and perpendicular to $T_j[0,0]$ while 
	$$ V_j^\sharp := \frac{V_j - p_0[0] \langle p_0[0] , V_j \rangle}{\sin(\tau_j + 2 \alpha)} $$
	is the re-scaled orthogonal projection of $V_j$ into $T_{p_0[0]} \Sph^{n+1}$. 
\end{prop}

\begin{proof}

The various terms in the formula \eqref{eqn:balmapdependence} for $\mathring B(\vec \sigma)$ must be differentiated at $\vec \sigma = 0$.  Let $\vec \sigma_0 (t) = t V_0$ and $\vec \sigma_j (t) = t V_j$ be paths in the displacement parameter space, where $V_0$ and $V_j$ are considered as vectors in $T_{p_0[0]} \Sph^{n+1}$ and $T_{p_j[0]} \Sph^{n+1}$ respectively.  First, 
\begin{align*}
	\left. \frac{\dif}{\dif t} \right|_{t=0} \eps_j ( t V_0, t V_j) &=  \dot \eps(\tau_j) \cdot \left. \frac{\dif}{\dif t} \right|_{t=0} \arccos \bigl (\langle  p_0[t V_0] , p_j[t V_j]\rangle \bigr) \\[1ex]
	&= - \frac{\dot \eps_j(\tau_j) \cdot \bigl( \langle V_0, p_j[0]  \rangle + \langle p_0[0], V_j \rangle  \bigr)}{\sqrt{ 1 - \langle p_0[0], p_j[0] \rangle^2}} \, .
\end{align*}
The first term in the formula for $\Dif  \mathring B(0) (V)$ involving the parallel parts of $V_0$ and $V_j$ follows from this using the formula for $T_j[0,0]$ as well as $\langle p_0[0], V_0 \rangle = \langle p_j[0], V_j \rangle = 0$. .

Next, realize that $\big( \mathcal W_{\vec \sigma_0} \big)_\ast^{-1} T_j [\vec \sigma_0, \vec \sigma_j]$ is the tangent vector of the geodesic connecting the point $p_0[0]$ to $\mathcal W_{\vec \sigma_0} ^{-1} \circ \mathcal W_{\vec \sigma_j} (p_j[0])$ at $p_0[0]$.  A calculation reveals
$$\big( \mathcal W_{\vec \sigma_0} \big)_\ast^{-1} T_j [\vec \sigma_0, \vec \sigma_j] = \frac{\mathcal W_{\vec \sigma_0}^{-1} \circ \mathcal W_{\vec \sigma_j} (p_j[0]) -\big\langle \mathcal W_{\vec \sigma_0} ^{-1} \circ \mathcal W_{\vec \sigma_j} (p_j[0]) \, , \,  p_0[0] \big\rangle \cdot p_0[0] }{\sqrt{1 - \big\langle \mathcal W_{\vec \sigma_0}^{-1} \circ \mathcal W_{\vec \sigma_j} (p_j[0]) \,  , \,   p_0[0]  \big\rangle^2}} \, .$$
Together with the definition of $\mathcal W_{\vec \sigma_j}$ one then finds after some work
\begin{equation*}
	\begin{aligned}
		\left. \frac{\dif}{\dif t} \right|_{t=0} \big( \mathcal W_{t V_0} \big)_\ast^{-1} T_j [ t V_0,t V_j]
		&= \frac{T_j[0,0] \cdot \big( \langle V_j, p_0[0] \rangle + \langle V_0, p_j[0] \rangle \big) \cdot \langle p_j[0], p_0[0] \rangle }{1 -  \langle p_j[0], p_0[0]\rangle^2}  \\[1ex]
		&\qquad + V_j^\sharp -  \frac{V_0 \cdot  \langle p_j[0], p_0[0] \rangle}{\sqrt{ 1 -  \langle p_j[0], p_0[0]\rangle^2}}   \, .
	\end{aligned}
\end{equation*}
The second term in the formula for $\Dif  \mathring B(0) (V)$ involving the transverse parts of $V_0$ and $V_j^\sharp$ follows from this.
\end{proof}

\subsection{Conclusion of the Proof of Main Theorem 1}

The ordinary inverse function theorem for smooth functions will be used to locate a zero of $B_\tau$.  The first step is to approximate $B_\tau$ by the simpler mapping $\mathring B_\tau : \mathcal D_{\Gamma} \rightarrow \R^{K}$ obtained by replacing each $B_j$ term in \eqref{eqn:balmap} by the corresponding function $\mathring B_j : \mathcal D_{\Gamma} \rightarrow \R$ of the form \eqref{eqn:balmapapprox}.  The mapping $\mathring B_\tau$ is independent of $f_{ \tau, \vec \sigma}$ and therefore depends only on the geometry of initial configuration $\Lambda^{\#}[\alpha, \Gamma, \tau, \vec\sigma]$.    The hypotheses of Main Theorem 1 assert that $\Lambda^{\#}[\alpha, \Gamma, \tau, \vec\sigma]$ is balanced, meaning that $\mathring B_\tau(0) = 0$.  By Proposition \ref{prop:balmapapprox}, one now has $B_{\tau}(0) = E_\tau (0)$ where $E_\tau(\vec \sigma) := B_\tau(\vec \sigma) - \mathring B_\tau (\vec \sigma)$.  This error term satisfies $\| E_\tau(0) \| = \mathcal O (r_\eps^{n})$ which is smaller than the operator norm of $\Dif \mathring B(0)$.   one can therefore attempt to use the finite-dimensional inverse function theorem to find a nearby $\vec \sigma$ so that $B_{\tau}(\vec \sigma) = 0$. 

It is important to incorporate into the analysis the fact that $B_{\tau}$ can often not be a full-rank mapping.   To see this why this is so, let $Y_1, \ldots, Y_d$ be a basis for the infinitesimal generators of one-parameter families of rotations of $\Sph^{n+1}$ that are equivariant with respect to the symmetries of $\approxsol$.  This means $\rho_\ast ( Y_j \circ \rho)  = Y$ for all $\rho \in G_{\Gamma}$ and this ensures that the functions $\langle Y_j, \nu \rangle : \approxsol \rightarrow \R$, where $\nu$ is the outward unit normal of $\approxsol$, are invariant with respect to $G_{\Gamma}$.  Now the first variation formula for the volume of hypersurfaces, applied to the volume-preserving deformation given by rotation in the $Y_j$ direction,  leads to the equation  
$$\int_{\phi_f(\approxsol)} \difop(f) \cdot \langle \nu, Y_j \rangle = 0 \qquad \quad \forall \: \: j = 1, \ldots, d$$
where $\nu_f$ is the unit outward normal vector field of $\phi_f(\approxsol)$.  Therefore one sees that there are maps $\mathcal Y_j : \mathcal D_{\Gamma} \rightarrow \R^K$ for $j= 1, \ldots, d$ with 
\begin{equation}
	\label{eqn:maxrank1}
	\big[ B_\tau (\vec \sigma) \big] \cdot \big[ \mathcal Y_j (\vec \sigma) \big] = 0 \qquad \quad \forall \: \: j = 1, \ldots, d
\end{equation}
where $\cdot$ denotes the Euclidean inner product.  Hence the rank of $B_\tau$ is at most $K-d$.   

The correct interpretation of \eqref{eqn:maxrank1} is to say that the graph $\{ \big(\vec \sigma, B_\tau (\vec \sigma) \big) : \vec \sigma \in \mathcal D_{\Gamma} \}$ is contained in the submanifold $\{ ( \vec \sigma, b ) : b \cdot \mathcal Y_1(\vec \sigma)  = \cdots = b \cdot \mathcal Y_d(\vec \sigma) = 0  \}$ of $\mathcal D_\Gamma  \times \R^K$.  Therefore it suffices to show that the equation $\mathit{pr} \circ B_\tau (\vec \sigma) = 0$ has a solution, where $\mathit{pr}$ is the orthogonal projection to the orthogonal complement of the subspace spanned by $\mathcal Y_1(0), \ldots, \mathcal Y_d(0)$.  Note that the linearization of $\mathit{pr} \circ B_\tau$ at zero maps \emph{into} this orthogonal complement, and thus $\Dif \big( \mathit{pr} \circ B_\tau \big) (0) = \Dif  B_\tau  (0)$.  In addition, the calculations of the proof of Proposition \ref{prop:balmapapprox} show that $(\mathit{id} - \mathit{pr}) \circ \Dif \mathring B_\tau(0) = L$ where $L$ is a linear operator with $\mathcal O (r_\eps^n)$ coefficients.

The hypotheses of Main Theorem 1 assert that $\Lambda^{\#}[\alpha, \Gamma, \tau, \vec\sigma]$ has the property that $\Dif \mathring B_\tau(0)$ has full rank.  Hence $(\mathit{id} - \mathit{pr}) \circ \Dif \mathring B_\tau(0)$ and $\Dif B_\tau(0)$ do as well.  Furthermore, the operator norm of $\Dif \mathring B_\tau(0)$ is $\mathcal O(C(\eps) \eps^{n-1})$.  Hence  $B_\tau(\vec \sigma) = b$ will be solvable for $b$ inside a ball centered on  $\mathit{pr} \circ E_\tau (0)$ whose radius is  $\mathcal O(C(\eps) \eps^{n-1})$.  When $\eps$ is sufficiently small, $0$ is contained within this ball.  Hence the equation $B_\tau(\vec \sigma ) = 0$ is solvable for small $\vec \sigma \in \mathcal D_{\Gamma}$.  The proof of Main Theorem 1 is therefore complete. \hfill \qedsymbol

\section{Applications of the Balancing Formul\ae}
\label{sec:examples}

\subsection{A Simple Example}
\label{sec:simple}

A simple example serves both to develop intuition for the approximate balancing map \eqref{eqn:balmapapprox} and its derivative \eqref{eqn:balmapderiv}, as well as to show that the kernel of the derivative of the approximate balancing map can be quite large in the absence of symmetries.  While this feature is also present in the CMC gluing construction in Euclidean space, it is here much more restrictive because the trick of imposing decay conditions at infinity to reduce the size of the kernel of the Euclidean analogue of the approximate balancing map is not available.  Therefore one must impose symmetry conditions or else expect to work quite hard to find an initial configuration of hyperspheres that can be glued together and perturbed into an exactly CMC hypersurface using the gluing technique. 

Consider exactly one geodesic, without loss of generality the $(x^0, x^1)$-equator $\gamma$, and let $R^{01}_\theta$ be the rotation by an angle $\theta$ in the $(x^0, x^1)$-plane that translates along $\gamma$.  Position $N$ hyperspheres of radius $\cos(\alpha)$ around $\gamma$, separated by a distance of $\tau$ from each other, so that $(\tau + 2 \alpha) N = 2 \pi m$ for some integer $m$.  These hyperspheres are of the form $S^k_\alpha := \big( R^{01}_{\tau + 2 \alpha)} \big)^k(S_\alpha)$ which are centered at $p_k := \gamma((\tau + 2 \alpha)k)$.  Let $\Lambda^\# := \bigcup_{k=0}^{N-1} S^k_\alpha $.   Note that this initial configuration is balanced because the vanishing of the approximate balancing map is equivalent to the equal spacing of the hyperspheres along a single geodesic.  

The initial configuration $\Lambda^\#$ yields the Delaunay-like hyperspheres in Butscher's paper \cite{me5} using the gluing technique together with imposing as many symmetries as possible on the deformations.  Now, however, no symmetries will be imposed and as a result the approximate balancing map becomes non-trivial. In the absence of any symmetry conditions constraining the displacement parameters of $\Lambda^\#$, there are $n$ displacement parameters for each hypersphere in $\Lambda^\#$.  For each hypersphere $S^k_\alpha$, these will be decomposed into one displacement parameter  corresponding to the displacement of $S^k_\alpha$ along $\gamma$ and $n-1$ displacement parameters  corresponding to the displacement of $S^k_{\alpha}$ perpendicular to $\gamma$.  To parametrize these displacement parameters in a uniform way, note that $T_{p_k} \Sph^{n+1}$ is spanned by $T_k := \dot \gamma ((\tau + 2 \alpha)k )$ and $\frac{\partial}{\partial x^2} , \ldots, \frac{\partial}{\partial x^n}$.  Thus one can set 
$$\vec \sigma^k :=  \sigma^k_1 \, T_k+ \sum_{j=2}^n \sigma^k_j \frac{\partial}{\partial x^j}  $$
as the displacement parameter for $S^k_\alpha$.  Note that $\vec \sigma^0 = \vec \sigma^N$ by periodicity.

It will now be shown that the kernel of the derivative of the approximate balancing map is very large.  Note that Main Theorem 1 still applies because each element in the kernel of $\Dif \mathring B (0)$ is induced from a rotation of $\Sph^{n+1}$.  Let $\vec V^k := ( V^k_1, V^k_2, \ldots, V^k_n)$ denote infinitesimal displacements satisfying $\vec V^0 = \vec V^N$.  To compute $\Dif \mathring B (0, \ldots, 0) (\vec V^1, \ldots , \vec V^N)$ one needs formul\ae\ for the re-scaled orthonormal projection operators $X \mapsto X^{\sharp}$ that appear there.  It is easy to deduce
\begin{align*}
	( V^{k \pm 1}_1 )^{\sharp_{k}} &= \frac{V^k_1}{\tan(\tau + 2 \alpha)} \, T_k \\[1ex]
	( V^{k \pm 1}_j )^{\sharp_{k}} &= \frac{V^k_j}{\sin(\tau + 2\alpha)}\, \frac{\partial }{\partial x^j} \qquad j= 2, \ldots, n
\end{align*}
Consequently, the derivative of the approximate balancing map on the $k^{\mathit{th}}$ perturbed hypersphere takes the form
\begin{equation*}
	\begin{aligned}
		\Dif \mathring B (0) (\vec V^{k-1}, \vec V^k, \vec V^{k+1}) &= - \omega \eps^{n-2} \dot \eps \big( 2 V_1^k - V_1^{k-1} - V_1^{k+1} \big) \\
		&\qquad - \omega \eps^{n-1} \big( 2 V_j^k - \sec(\tau + 2 \alpha) (V_j^{k-1} + V_j^{k+1}) \big) \, .
	\end{aligned}
\end{equation*}
The recursion formul\ae\ $2 V_1^k - V_1^{k-1} - V_1^{k+1} = 0$ and $2 V_j^k - \sec(\tau + 2 \alpha) (V_j^{k-1} + V_j^{k+1}) = 0$ for elements in the kernel of $\Dif \mathring B (0)$, together with the periodic boundary conditions $V^0_j = V^N_j$ for all $j = 1, \ldots, N$, are easy to solve and yield
\begin{align*}
	V^k_1 &= 1 \qquad \mbox{for all} \: \: k = 0, \ldots, N \\
	V^k_j &= \sin \big( (\tau + 2 \alpha)(k_0 + k) \big) \qquad \mbox{for} \: \: j = 2, \ldots, n \: \: \mbox{and} \: \: k_0 \in \{ 0, \ldots, N-1\} \, .
\end{align*}
There are either $\frac{1}{2} N(n-1) +1$ or $N(n-1)+1$ linearly independent solutions of this type, depending on whether $N$ is even or odd.  These solutions correspond to the change of displacement parameter induced by the rotation of $\Sph^{n+1}$ parallel to $\gamma$ and transverse to $\gamma$.

\subsection{An Unachievable Configuration}
\label{sec:nonexample}

The intuition gained from the preceding example can be used to explain why a reasonably simple configuration, possessing an analogue in Euclidean space, cannot be achieved using the gluing technique.  The configuration in question consists of positioning hyperspheres around two intersecting geodesics that make an arbitrary to each other at the point of intersection.  This is a slightly less symmetric version of the configuration considered in \cite{me5} where a CMC hypersurface is created from hyperspheres positioned around two orthogonally intersecting geodesics.

The reason the less symmetric configuration can not be glued together and perturbed into a CMC hypersurface is the following.  First, let $R^{01}_\theta$ be the rotation of the $(x^0, x^1)$-plane as before and let $R^{02}_\theta$ be the rotation of the $(x^0, x^2)$-plane.  Let $\gamma_j$ be the $(x^0, x^j)$-equator for $j=1,2$.   Choose $\alpha, \tau \in (0, \pi)$ and integers $m, N$ so that $(\tau + 2 \alpha)N = 2 \pi m$.  Also, choose $N$ of the form $N = 4 N_0$.  The initial configuration in question, which shall be denoted $\Lambda^\#_\theta$, consists of the hyperspheres $S^{2,k, \pm}_\alpha := R^{01}_{\pm\theta} \circ ( R^{02}_{\tau + 2 \alpha})^k(S_\alpha)$ for $k= 0, \ldots, N-1$.  When $\theta \neq \pi/2$, the maximal symmetries one can impose on the deformations of the approximate solution constructed from $\Lambda^\#_\theta$ are: all orthogonal transformations of the $x^3, \ldots, x^{n+1}$ coordinates; and the reflections sending $x^j$ to $- x^j$ and keeping all other coordinates fixed, for $j = 0, 1, 2$.   As a result, there are two sets of invariant approximate Jacobi fields.  These are: the translation of $S^{2,k,+}_\alpha$ along the geodesic $R^{01}_\theta(\gamma_2)$ for a given $k \in \{ 1, \ldots, N_0 - 1\}$ and then extended by symmetry to $S^{2, k, -}_\alpha$, $S^{2, -k, \pm}_\alpha$ and $S^{2, 2N_0 \pm k, \pm}_\alpha$ by symmetry; and the rotation of $S^{2, k,+}_\alpha$ in the $(x^0, x^1)$-plane transverse to $R^{02}_\theta(\gamma)$ and similarly extended by symmetry.  Furthermore, none of these invariant approximate Jacobi fields are induced by rotations of $\Sph^{n+1}$.  Note that there are no invariant approximate Jacobi fields associated to the hyperspheres $S^{2, 0,\pm}_\alpha$, $S^{2, N_0,\pm}_\alpha$, $S^{2, 2N_0,\pm}_\alpha$ and $S^{2, 3N_0,\pm}_\alpha$.

In order to glue together the initial configuration described above and to perturb it into a CMC hypersurface, it would be necessary to apply the balancing arguments to deal with the invariant approximate Jacobi fields.  Clearly $\Lambda^\#_\theta$ is balanced for each $\theta$ because the separation parameters between all hyperspheres are equal and its geodesic segments meet in parallel pairs.  Thus it would remain to check only that the derivative of the approximate balancing map has full rank (which corresponds to being invertible in this case because the imposed symmetries rule out all co-kernel coming from induced rotations of $\Sph^{n+1}$).  However, the analysis of the simple example of Section \ref{sec:simple} shows that the kernel of $\Dif \mathring B(0)$ is one-dimensional and consists of the transverse motion $V^k = \sin((\tau + 2 \alpha)(N_0 + k) \frac{\partial}{\partial x^1}$ and extended by symmetry.  This approximate Jacobi field is induced by the change of the $\theta$-parameter and not by a rotation of $\Sph^{n+1}$.  Therefore Main Theorem 1 does not apply to $\Lambda^\#_\theta$ unless $\theta = \pi/2$, in which case there is an additional symmetry (invariance with respect to the rotation $R^{01}_{\pi/2}$) that eliminates this approximate Jacobi field from consideration.

\begin{rmk}
	The analogue of the example above in Euclidean space consists of two Delaunay surfaces with non-parallel axes meeting at a common spherical region.  It is possible to glue this initial configuration together and perturb it into a CMC hypersurface.  This is because the decay conditions at infinity that are built into the function space used in the analysis rules out the approximate Jacobi fields corresponding to the change-of-angle parameter and the translation parameter. 
\end{rmk}

\subsection{A Related Achievable Configuration}
\label{sec:achievable}

A modification of the previous example yields an initial configuration of hyperspheres to which Main Theorem 1 does apply.  The key is to `freeze' the motion of the $\theta$-parameter without imposing additional symmetries, which can be achieved by adding another set of spheres along the geodesic orthogonal to the initial configuration of Section \ref{sec:nonexample}.  The requirement that the spheres at the intersection points of the geodesic with the initial configuration of Section \ref{sec:nonexample} match perfectly is what freezes the motion in the $\theta$-parameter.  That is, choose an integer $k_0$ and let 
$$\Lambda^\# := \left[  \bigcup_{k=0}^{N-1} S^{2, k, +}_\alpha \cup S^{2, k, -}_\alpha \right] \cup \left[  \bigcup_{k=0}^{N-1}  S^{1, k}_{\alpha} \right] $$
where $S^{1, k}_\alpha := (R^{01}_{\tau + 2 \alpha})^k(S_\alpha)$.  Note that $\Lambda^\#$ has the same group of symmetries as before.  Its approximate Jacobi fields are those described before as well as the translation of $S^{1,k}_\alpha$ along the geodesic $\gamma_1$ for any given $k \in \{ 1, \ldots, N_0 - 1\}$ and extended by symmetry.  Again, there are no approximate Jacobi fields associated to the hyperspheres $S^{1,0}_{\alpha}$, $S^{1,N_0}_{\alpha}$, $S^{1,2 N_0}_{\alpha}$ and $S^{1,3N_0}_{\alpha}$.

The initial configuration $\Lambda^\#$ is balanced because the separation parameters between all hyperspheres are equal and its geodesic segments meet in parallel pairs.  Thus to apply Main Theorem 1 it remains to check that the derivative of the approximate balancing map is invertible.  Let $T_{1,k} := \dot \gamma_1 ((\tau + 2 \alpha)k)$ and $T_{2, k, \pm} := \big( R^{01}_{\pm (\tau + 2 \alpha )k_0)} \big)_\ast \dot \gamma_2 ((\tau + 2 \alpha)k)$ be the tangent vectors of the geodesics $\gamma_1$ and $\gamma_2$ at the centers of the hyperspheres of $\Lambda^\#$ and define
\begin{equation*}
	\begin{aligned}
		\vec V^{1,k} &:= u^k T_{1, k} \\[1ex]
		\vec V^{2,k, \pm} &:=  v^{k, \pm} T_{2, k, \pm}  + w^{k, \pm} \frac{\partial}{\partial x^1}
	\end{aligned}
\end{equation*}
as the displacement parameters of these hyperspheres.  Note that 
\begin{align*}
	&u^k = v^{k, \pm} = 0 \qquad k \equiv 0 \mod 4 \\
	&w^{0, \pm} = w^{2N_0, \pm} = 0 \\
	&u^k = -u^{-k} = - u^{2N_0 + k} = u^{2N_0 - k} \qquad k = 1, \ldots, N_0-1
\end{align*}
and similarly for $v^{\ast}$ and $w^\ast$ by symmetry.  In addition $u^{k_0} = w^{0,+}$ since the corresponding hyperspheres coincide.  Thus it is only necessary to analyze the action of $\Dif \mathring B(0)$ on the vector $\vec V := (\vec V^{1,1}, \ldots, \vec V^{1, N_0 - 1}, \vec V^{2, 1, +}, \ldots, \vec V^{2, N_0 - 1, +})$ and set $v^k := v^{k,+}$ and $w^k := w^{k, +}$.  One finds
\begin{equation*}
	\Dif \mathring B(0)(\vec V) := \left( 
	\begin{array}{c}
		\vdots \\
		- (n-1) \omega \eps^{n-2} \dot \eps \big(2 u^k - u^{k+1} - u^{k-1} \big) \\
		- (n-1) \omega \eps^{n-2} \dot \eps \big(2 v^k - v^{k+1} - v^{k-1} \big) \\
		-  \omega \eps^{n-1} \big(2 w^k - w^{k+1} - w^{k-1} \big) \\
		\vdots
	\end{array}
	\right)  \, .
\end{equation*}
If the equations in $\Dif \mathring B(0)(\vec V) = 0$ were uncoupled, then the kernel would be of the form found in the simple example of Section \ref{sec:simple}.  If the boundary conditions are included, then it follows that $v^k = 0$  for all $k$, as well as $u^k = c$ for all $k$ and $w^k = c' \sin((\tau + 2 \alpha)(k+N_0))$ for $c, c' \in \R$.  But the coupling $2 w^0 - 2 \sec(\tau + 2 \alpha) w^1 = 2 u^{k_0} - 2 \sec(\tau + 2 \alpha) w^1 = 0$ then forces $c = c' = 0$.  Hence $\Dif \mathring B(0)$ is invertible and Main Theorem 1 applies to allow $\Lambda^\#$ to be glued together and perturbed into a CMC hypersurface.

\subsection{An Achievable Configuration Without Any Symmetries}

The previous example has much less symmetry than the examples constructed in \cite{me5} but still possesses a large symmetry group.  Further modifications of the ideas of the previous sections leads to examples of initial configurations to which Main Theorem 1 applies with few symmetries or no symmetries at all.  These example are naturally quite hard to write down, and in any case the purpose of this final section of the paper is to give the reader the necessary ideas for constructing these examples, so it is sufficient to proceed in the $n=2$ case.

The first modification leading to a much less symmetric example is to consider $\Lambda^\#$ from Section \ref{sec:achievable}, except with the new geodesic tilted into the $x^3$-direction by some an angle which is not $\pi/2$.  Such an example would still be balanced because its geodesic segments would continue to meet in parallel pairs.  Also, such an example would clearly possess no symmetries other than the $x \mapsto - x$ reflection sending a point on $\Sph^3$ to the antipodal point.  However, it is not immediately clear that it is possible to tilt the third geodesic so that equally spaced hyperspheres of radius $\cos(\alpha)$ along the third geodesic line up exactly with the hyperspheres of the same radius along the first two geodesics where these geodesics meet.  But a moment's thought reveals that what is needed for some configuration of equally spaced spheres of some radius winding some perhaps large number of times around $\Sph^3$ to exist is that all the geodesic segments have lengths which are rational multiples of $2 \pi$.  This, in turn, can be achieved if the three unit vectors $N_1, N_2, N_3$ orthogonal to the planes containing the three geodesics have $\langle N_i, N_j  \rangle \in 2 \pi \Q$ for all $i, j \in \{ 1, 2, 3\}$.  This can be achieved.  The details of the balancing arguments that prove that Main Theorem 1 applies are identical to the arguments of Section \ref{sec:achievable} and thus the configuration above can be glued together and perturbed into a CMC hypersurface.

One final modification of these ideas leads to an example without any symmetries at all.  The idea is to perform the same trick of adding in a tilted geodesic to a configuration which does not have the $x \mapsto -x$ antipodal symmetry.  Such a configuration is the following: consider three half-geodesics of the form $R^{02}_{2 \pi/3} \big( \gamma_1( [0, \pi]) \big)$ and choose a fourth geodesic which is tilted into the $x^3$-direction.  The reader can verify that the fourth geodesic can be chosen in such that equally positioned hyperspheres match appropriately and that the balancing arguments needed to apply Main Theorem 1 hold.  Hence this configuration can be glued together and perturbed in a CMC hypersurface as well.

\bibliography{sphcmc}
\bibliographystyle{amsplain}

\end{document}